\documentclass{article}
\usepackage[latin1]{inputenc}
\usepackage{amsmath,amsthm,amssymb}
\usepackage{amsfonts}
\usepackage{amsmath,amsthm,amssymb,amscd}
\usepackage{latexsym}
\usepackage{color}
\usepackage{graphicx}
\usepackage{mathrsfs}

\makeatletter \@addtoreset{equation}{section} \makeatother

\newtheorem{theorem}{Theorem}[section]
\newtheorem{definition}{Definition}[section]

\newtheorem{lemma}{Lemma}[section]

\begin{document}
\title{Existence and symmetry of solutions for critical fractional \\ Schr\"{o}dinger equations with bounded potentials }
\author{Xia Zhang$^a$, Binlin Zhang$^{b, }$\footnote{Corresponding author.
E-mail address:\,piecesummer1984@163.com (X. Zhang), zhangbinlin2012@163.com(B. Zhang), dusan.repovs@guest.arnes.si (D. Repov\v{s})}\ \ and\ Du\v{s}an Repov\v{s}$^{c}$\\
\footnotesize $^a$ Department of Mathematics, Harbin Institute of Technology,
Harbin 150001, P.R. China\\
\footnotesize $^b$ Department of Mathematics, Heilongjiang Institute of Technology,
Harbin 150050, P.R. China\\
\footnotesize $^c$ Faculty of Education and Faculty of Mathematics and Physics,
University of Ljubljana,\\ \footnotesize
Kardeljeva plo\v{s}\v{c}ad 16, SI-1000 Ljubljana, Slovenia
}
\date{ }
\maketitle

\begin{abstract}
{This paper is concerned with the following fractional Schr\"{o}dinger equations involving critical exponents:
\begin{eqnarray*}
(-\Delta)^{\alpha}u+V(x)u=k(x)f(u)+\lambda|u|^{2_{\alpha}^{*}-2}u\quad\quad
\mbox{in}\ \mathbb{R}^{N},
\end{eqnarray*}
where $(-\Delta)^{\alpha}$ is the fractional Laplacian operator with $\alpha\in(0,1)$,
$N\geq2$, $\lambda$ is a positive real parameter and $2_{\alpha}^{*}=2N/(N-2\alpha)$ is the critical Sobolev exponent,
$V(x)$ and $k(x)$ are positive and bounded functions satisfying some extra hypotheses.
Based on the principle of concentration compactness  in the fractional  Sobolev space and the minimax arguments, we obtain the existence of a  nontrivial  radially symmetric weak solution for the above-mentioned equations without assuming the Ambrosetti-Rabinowitz
condition on the subcritical nonlinearity.}\medskip

\emph{\bf Keywords:}   fractional Schr\"{o}dinger equations;  critical Sobolev exponent; Ambrosetti-Rabinowitz
condition;  concentration compactness principle\medskip

\emph{\bf 2010 MSC:} 35A15, 35J60, 46E35.
\end{abstract}

\section{Introduction and main result}

\quad In this paper,  we study the solutions of the  following Schr\"{o}dinger equations involving a critical nonlinearity:
\begin{eqnarray}\label{p}
(-\Delta)^{\alpha}u+V(x)u=k(x)f(u)+\lambda|u|^{2_{\alpha}^{*}-2}u\quad\quad
\mbox{in}\ \mathbb{R}^{N},
\end{eqnarray}
driven by the fractional Laplacian operator $(-\Delta)^{\alpha}$ of order $\alpha\in(0,1)$,
where $N\geq2$,    $\lambda$ is a positive real parameter and $2_{\alpha}^{*}=2N/(N-2\alpha)$ is the critical Sobolev exponent.

The fractional Laplacian operator $(-\Delta)^{\alpha}$, which (up to normalization constants),  may be defined as
$$(-\Delta)^{\alpha}u(x):= \mbox{P.V.} \int_{\mathbb{R}^{N}}\frac{u(x)-u(y)}{|x-y|^{N+2\alpha}}\,dy,\ \ \ x\in\mathbb{R}^{N},$$
where P.V. stands for the principal value.
It may be viewed as the infinitesimal generators of a
L\'{e}vy stable diffusion processes (see \cite{Apple2004}). This operator arises in the description of various phenomena in the applied sciences,  such as phase transitions,  materials science, conservation laws, minimal surfaces, water waves, optimization, plasma physics and so on,
see \cite{Nezza} and references therein for more detailed introduction. Here we would like to point out some interesting models involving the fractional Laplacian, such as, the fractional Schr\"{o}dinger equation (see \cite{Dipierro2013, Felmer, Laskin2000, L2, MRa}), the fractional Kirchhoff equation (see \cite{FV, PXZ, PXZ2, XZF1, XZG}), the fractional porous medium equation (see \cite{CV, JLV}),
the fractional Yamabe problem (see \cite{PPS}) and so on, have attracted recently considerable attention.  As a matter of fact, the literature on fractional operators and their applications to partially differential equations is quite large, here we would like to mention a few, see for instance \cite{AP, Cabre,  Chang2013, MS, MRe, Ros-Oton}.

In what follows, let us sketch the related advance involving the fractional Schr\"{o}dinger equations with critical growth in resent years.
In \cite{SZ}, Shang and Zhang studied the existence and multiplicity of solutions for the critical fractional Schr\"{o}dinger equation:
\begin{eqnarray}\label{L1}
\varepsilon^{2\alpha}(-\Delta)^{\alpha}u+V(x)u=\lambda f(u)+|u|^{2_{\alpha}^{*}-2}u\quad\quad
\mbox{in}\ \mathbb{R}^{N}.
\end{eqnarray}
Based on variational methods, they showed that problem \eqref{L1} has a nonnegative ground state solution for all sufficiently large $\lambda$ and small $\varepsilon$. In this paper, the following monotone condition was imposed on the continuous subcritical nonlinearity $f$:
\begin{align}\label{L2}
f(t)/t\ \mbox{is strictly increasing in}\ (0,+\infty).
\end{align}
Observe that \eqref{L2} implies  $2 F(t)< f(t)t$, where $F(t):=\int_0^t f(\xi)\,d\xi$. Moreover, Shen and Gao in \cite{SG} obtained the existence of nontrivial solutions for problem \eqref{L1}
under various assumptions on $f(t)$ and potential function $V(x)$, in which the authors assumed
the well-known Ambrosetti-Rabinowitz
condition ((AR) condition for short) on $f$:
\begin{align}\label{L3}
\mbox{there exists}\ \mu>2\ \mbox{such that}\ 0<\mu F(t)\leq f(t)t\ \,\mbox{for\ any}\   t>0.
\end{align}
 See also recent papers \cite {SZY1, SS1} on the fractional Schr\"{o}dinger equations \eqref{L1}.
In \cite{KT}, Teng and He were concerned with the following fractional Schr\"{o}dinger equations involving a critical nonlinearity:
\begin{eqnarray}\label{L4}
(-\Delta)^{\alpha}u+u=P(x)|u|^{p-2}u+Q(x)|u|^{2_{\alpha}^{*}-2}u\quad\quad
\mbox{in}\ \mathbb{R}^{N}.
\end{eqnarray}
where $2<p<2^{*}_{\alpha}$, potential functions $P(x)$ and $Q(x)$ satisfy certain hypotheses.
Using the $s$-harmonic extension technique of Caffarelli and Silvestre \cite{CS},
the concentration-compactness principle of Lions \cite{PL} and methods of Br\'{e}zis and Nirenberg \cite{BN},
the author obtained the existence of ground state solutions.
On fractional Kirchhoff problems involving critical nonlinearity, see for example \cite{AFP, Pucci} for some recent results.
Last but not least, fractional elliptic problems with critical growth, in a bounded domain, have been studied by some authors in the last years,
see \cite{Barrios,BCSS,  Hua, MS1, SV1, SV4} and references therein.

On the other hand, Feng in \cite{Feng} investigated the following fractional Schr\"{o}dinger equations:
\begin{eqnarray}\label{L5}
(-\Delta)^{\alpha}u+V(x)u=\lambda|u|^{p-2}u\quad\quad
\mbox{in}\ \ \mathbb{R}^{N},
\end{eqnarray}
where $2<p<2^{*}_{\alpha}$, $V(x)$ is a positive continuous function.
By using the fractional version of concentration compactness principle of Lions \cite{PL},
the author obtained the existence of ground state solutions to problem \eqref{L5} for some $\lambda>0$.
Zhang {\em et al.} in \cite{ZZX} considered the following fractional Schr\"{o}dinger equations with a critical nonlinearity:
\begin{eqnarray}\label{Eq}
(-\Delta)^{\alpha}u+u=\lambda f(u)+|u|^{2_{\alpha}^{*}-2}u\quad\quad
\mbox{in}\ \mathbb{R}^{N}.
\end{eqnarray}
Based on another fractional version of concentration compactness principle (see \cite[Theorem 1.5]{Palatucci2014})
and radially decreasing rearrangements,
they obtained the existence of a ground state solution for \eqref{Eq} which is nonnegative and radially symmetric for any $\lambda\in[\lambda_{*},\infty)$, where $\lambda_{*}>0$.

Inspired by the above works, we are interested in non autonomous cases \eqref{p}, that is, $V(x)$ is not only a constant.
To this end, we assume the following conditions on the potential $V$:

{\em\begin{itemize}
\item[${\rm (V1)}$] $V\in C^1(\mathbb{R}^N,\mathbb{R})$ and  $ \nabla V(x)\cdot x\leq 0$ for any $x\in\mathbb{R}^N$;
\item[${\rm (V2)}$] $V$ is  radially symmetric, i.e.  $V(x)=V(|x|)$ for any $x\in\mathbb{R}^N$ and there exist positive constants $V_{1}$ and  $V_{2}$ such that $V_{1}\leq V(x)\leq V_{2}$ for any $x\in\mathbb{R}^N$.
\end{itemize}}

Moreover, the following assumptions are imposed on the coefficient $k$:

{\em\begin{itemize}
\item[${\rm (K1)}$] $k$ is  radially symmetric and there exist positive constants $k_{1}$ and  $k_{2}$ such that $k_{1}\leq k(x)\leq k_{2}$ for any $x\in\mathbb{R}^N$;
\item[${\rm (K2) }$] $k\in C^1(\mathbb{R}^N,\mathbb{R})$ and there exists a constant $k_0$ such that $0\leq \nabla k(x)\cdot x\leq k_0$ for any $x\in\mathbb{R}^N$.
\end{itemize}}

\noindent\textbf{Remark 1.1} Since $\nabla V(x)\cdot x=V'(|x|)|x|$, it follows from  (V1)  that $V'(|x|)\leq0$.
Thus we can choose $V$ to be  a positive constant.
Another example for $V$ is given by $V(x)=2-\arctan|x|$.
From (K2), $k'(|x|)\geq0$. Hence we can choose $k(x)=2+\arctan|x|$ as a simple example. The condition (V1) and (K2) were motivated by \cite{Jeanjean, SS2}.

\vspace{2mm}

Meanwhile, the nonlinearity $f$ will satisfy:

{\em\begin{itemize}
\item[${\rm (H1)}$] $f\in C^1(\mathbb{R},\mathbb{R})$. For any $t\leq0$, $f(t)=0$;
\item[${\rm (H2)}$] $\lim_{t\rightarrow0^+}\frac{f(t)}{t}=0$ and $\lim_{t\rightarrow+\infty}\frac{f(t)}{t^{2_{\alpha}^{*}-1}}=0$;
\item[${\rm (H3)}$] For any $t>0$, $0<2F(t)\leq f(t)t$;
\item[${\rm (H4)}$] There exists $T>0$ such that $F(T)>\frac{V_{2}}{2k_{1}}T^2$.
\end{itemize}}

\noindent\textbf{Remark 1.2} In order to seek nonnegative solutions of (\ref{p}), we assume that  $f(t)=0$  for any $t\leq0$ in (H1). Moreover, from (H2) we know that $f$ is subcritical.
Here we do not assume classical condition \eqref{L2} or \eqref{L3},
while the weaker condition (H3) on $f$ is employed to replace (AR) condition.
A typical  example   for  $f$  is   given  by
$$f(t)=t\log\left[1+t\left(t^2-\frac{1}{2}t-\frac{3}{2}at+a\right)\right]$$
for any $t\geq0$ and a certain constant $a>1/3$ which is sufficiently close to $1/3$.
It is easy to see that the function $f$ does not fulfill the monotone condition \eqref{L2} and the (AR) condition \eqref{L3}.

\vspace{2mm}

Now we  give the definition of  weak solutions  for problem   (\ref{p}):

\begin{definition}\label{def}
{\em We say that
$u$ is a weak solution of \textup{(\ref{p})} if for any $\phi\in H^{\alpha}(\mathbb{R}^{N})$,
\begin{displaymath}
\begin{split}
\int_{\mathbb{R}^{N}}((-\Delta)^{\frac{\alpha}{2}}u\cdot(-\Delta)^{\frac{\alpha}{2}}\phi+V(x)u\phi)\,dx
=\int_{\mathbb{R}^{N}}(k(x)f(u)+\lambda|u|^{2_{\alpha}^{*}-2}u)\phi\,dx,
\end{split}
\end{displaymath}
where $H^{\alpha}(\mathbb{R}^{N})$ is the fractional Sobolev space, see Section 2 for more details.}
\end{definition}

The energy
functional on $H^{\alpha}(\mathbb{R}^{N})$ is defined as follows:
\begin{eqnarray*}
\begin{split}
I(u)=\frac{1}{2}\int_{\mathbb{R}^{N}} (|(-\Delta)^{\frac{\alpha}{2}} u|^{2}+V(x)u^{2})\,dx
-\int_{\mathbb{R}^{N}}k(x)F(u)\,dx-\frac{\lambda}{2_{\alpha}^{*}}\int_{\mathbb{R}^{N}}|u|^{2_{\alpha}^{*}}\,dx.
\end{split}
\end{eqnarray*}
It is easy to check that $I\in
C^{1}(H^{\alpha}(\mathbb{R}^{N}),\,\mathbb{R})$ and the critical point for $I$ is the weak solution of  problem (\ref{p}).
Let $O(N)$ be the group of orthogonal linear transformations in
$\mathbb{R}^{N}$. It is immediate  that $I$ is
$O(N)$-invariant. Then, by the principle of symmetric criticality of
Krawcewicz and Marzantowicz \cite{Kr1990}, we know that  $u_{0}$ is
a critical point of $I$ if and only if $u_{0}$ is a critical
point of
$$\widetilde{I}=I\big|_{H_{r}^{\alpha}(\mathbb{R}^{N})},$$
where
$$H^{\alpha}_{r}(\mathbb{R}^{N})=\left\{u\in H^{\alpha}(\mathbb{R}^{N}):\ u(x)=u(|x|)\right\},$$
is the fractional  radially symmetric Sobolev space.
Therefore, it suffices to prove the existence of
critical points  for $\widetilde{I}$ on $
H_{r}^{\alpha}(\mathbb{R}^{N})$.

Now we are in a position to state our main result as follows:

\begin{theorem}\label{Th1}
 Assume that hypotheses
\textup{(H1)--(H4)}, \textup{(V1)--(V2)} and \textup{(K1)--(K2)} are fulfilled. Then there exists $\lambda_{*}>0$
such that for any $\lambda\in(0,\lambda_{*})$, problem
\textup{(1.1)} has a nontrivial weak solution $u_{0}\in  H^{\alpha}(\mathbb{R}^{N})$ which is nonnegative and  radially symmetric.
\end{theorem}

\noindent\textbf{Remark 1.3} (i) In the proof of Theorem 1.1,  we follow an approximation procedure to obtain
a bounded (PS)  sequence $\{u_n\}$ for $\widetilde{I}$, instead of starting directly from an arbitrary (PS) sequence. To show the boundedness  of (PS) sequences $\{u_n\}$ for $\widetilde{I}$, we need condition (K2) on $k$. It allows us to make use of
a Pohozaev type identity to derive the boundedness  of $\{u_n\}$. A key
point which allows to use the identity is that $\{u_n\}$ is a sequence of exact critical points.
In fact, the requirement $f\in C^{1}(\mathbb{R},\mathbb{R})$ will only be used in the proof of Pohozaev identity.

(ii) To the best of our knowledge, there are only few papers that
study the existence and symmetry  of solutions for problem (1.1) by using concentration compactness principle in the fractional Sobolev space which is different from the version used in \cite{Feng}.

\vspace{3mm}

This paper is organized as follows. In Section 2, we will give some necessary definitions and properties of fractional Sobolev spaces. In Section 3, by using the principle of concentration compactness and minimax arguments, we give the proof of Theorem 1.1.

\section{The Variational Setting}\label{sec2}

\quad For the convenience of the reader, in this part we recall some definitions and basic properties of fractional Sobolev spaces $H^{\alpha}(\mathbb{R}^{N})$.
For a deeper treatment on these spaces and their applications to fractional Laplacian problems of elliptic type, we refer to \cite{Nezza, MRS} and references therein.

 For any $\alpha\in(0,1)$, the fractional Sobolev space $H^{\alpha}(\mathbb{R}^{N})$ is defined by
 $$H^{\alpha}(\mathbb{R}^{N})=\left\{u\in L^{2}(\mathbb{R}^{N}):[u]_{H^{\alpha}(\mathbb{R}^{N})}<\infty\right\},$$
 where $[u]_{H^{\alpha}(\mathbb{R}^{N})}$ denotes the so-called Gagliardo semi-norm, that is
  $$[u]_{H^{\alpha}(\mathbb{R}^{N})}
=\left(\iint_{\mathbb{R}^{2N}}\frac{|u(x)-u(y)|^{2}}{|x-y|^{N+2\alpha}}\,dxdy
 \right)^{1/2}$$
and $H^{\alpha}(\mathbb{R}^{N})$ is endowed with the norm
 $$||u||_{H^{\alpha}(\mathbb{R}^{N})}=[u]_{H^{\alpha}(\mathbb{R}^{N})}
 +||u||_{L^{2}(\mathbb{R}^{N})}.$$
As it is well known, $H^{\alpha}(\mathbb{R}^{N})$ turns out to be a Hilbert space with scalar product
$$\langle u,\ v\rangle_{H^{\alpha}(\mathbb{R}^{N})}=\iint_{\mathbb{R}^{2N}}\frac{(u(x)-u(y))(v(x)-v(y))}{|x-y|^{N+2\alpha}}\,dxdy
+\int_{\mathbb{R}^{N}}u(x)v(x)\,dx,$$
for any $u,v\in H^{\alpha}(\mathbb{R}^{N})$. The space $\dot{H}^{\alpha}(\mathbb{R}^{N})$ is defined as the completion of $C^{\infty}_{0}(\mathbb{R}^{N})$ under the norm $[u]_{H^{\alpha}(\mathbb{R}^{N})}$.

 By Proposition 3.6 in \cite{Nezza},  we have $[u]_{H^{\alpha}(\mathbb{R}^{N})}=||(-\Delta)^{\frac{\alpha}{2}}u||_{L^{2}(\mathbb{R}^{N})}$ for any $u\in H^{\alpha}(\mathbb{R}^{N})$, i.e.
 \begin{equation}\label{30}
\iint_{\mathbb{R}^{2N}}\frac{|u(x)-u(y)|^{2}}{|x-y|^{N+2\alpha}}\,dxdy=\int_{\mathbb{R}^{N}}|(-\Delta)^{\frac{\alpha}{2}}u(x)|^2\,dx.
 \end{equation}
 Thus,
 \begin{equation}\label{31}
\iint_{\mathbb{R}^{2N}}\frac{(u(x)-u(y))(v(x)-v(y))}{|x-y|^{N+2\alpha}}\,dxdy
 =\int_{\mathbb{R}^{N}}(-\Delta)^{\frac{\alpha}{2}}u(x)\cdot(-\Delta)^{\frac{\alpha}{2}}v(x)\,dx.
 \end{equation}

\begin{theorem}\label{th2.1} {\em (\cite[Lemma 2.1]{Felmer})} The embedding $H^{\alpha}(\mathbb{R}^{N})\hookrightarrow L^{p}(\mathbb{R}^N)$ is continuous for any $p\in [2,2_\alpha^*]$ and the embedding $H^{\alpha}(\mathbb{R}^{N})\hookrightarrow\hookrightarrow L_{loc}^p(\mathbb{R}^{N})$ is compact for any $p\in [2,2_\alpha^*)$.
\end{theorem}

\section{Proof of Theorem 1.1}\label{sec3}

\quad Throughout this section, we assume that conditions
(H1)--(H4), (V1)--(V2) and (K1)--(K2) are satisfied.
In this part, we will use minimax arguments and we denote that $C$ are $C_i$
are
positive constant, for any $i=1, 2\cdots$.

 A crucial step to obtain the existence of a critical point for $\widetilde{I}$ is to show the boundedness
of (PS) sequence. But it seems difficult under our assumptions. To overcome this
difficulty we use an indirect approach developed in \cite{Jeanjean1}.
For any $\eta\in[1/2,1]$, we consider the following  family of functionals defined  on $H^{\alpha}_r(\mathbb{R}^{N}):$
\begin{eqnarray*}
\begin{split}
I_{\eta}(u)=\frac{1}{2}\int_{\mathbb{R}^{N}} (|(-\Delta)^{\frac{\alpha}{2}} u|^{2}+V(x)u^{2})\,dx
-\eta\int_{\mathbb{R}^{N}}k(x)F(u)\,dx-\frac{\eta\lambda}{2_{\alpha}^{*}}\int_{\mathbb{R}^{N}}|u|^{2_{\alpha}^{*}}\,dx.
\end{split}
\end{eqnarray*}
It is easy to check that $I_{\eta}\in
C^{1}(H_r^{\alpha}(\mathbb{R}^{N}),\,\mathbb{R})$ and the critical point for $I_{\eta}$ is the weak solution of the following equation:
\begin{equation}\label{19}
(-\Delta)^{\alpha}u+V(x)u=\eta k(x)f(u)+\eta\lambda|u|^{2_{\alpha}^{*}-2}u\quad\quad
\mbox{in}\ \mathbb{R}^{N}.
\end{equation}

First, we will give the following two lemmas to show that $I_{\eta}$
has a Mountain Pass geometry.

\begin{lemma}\label{Le3.1} There exists $v_{0}\in H^{\alpha}_{r}(\mathbb{R}^{N})$ and $\overline{\eta}\in[1/2,1)$ such that $I_{\eta}(v_{0})<0$ for any $\eta\in[\overline{\eta},1]$,  where $v_{0}$ and $\overline{\eta}$ are independent of $\lambda$.
\end{lemma}

\begin{proof}
Let $R>0$, we define
\begin{eqnarray*}w(x)=
\begin{cases}T  &\mbox{for}\ |x|\leq R,\\
T(R+1-|x|) &\mbox{for}\ R<|x|<R+1,\\
0 &\mbox{for}\ |x|\geq R+1,
\end{cases}
\end{eqnarray*}
then $w\in H^{\alpha}_{r}(\mathbb{R}^{N})$. Hence, from (H4) we have
 \begin{equation*}
\begin{split}
&\int_{\mathbb{R}^{N}}\left(k_{1} F(w)-\frac{1}{2}V_{2}w^{2}\right)\,dx\\
=&\int_{B(0,R)}\left(k_{1} F(w)-\frac{1}{2}V_{2}w^{2}\right)\,dx+\int_{B(0,R+1)\setminus B(0,R)}\left(k_{1} F(w)-\frac{1}{2}V_{2}w^{2}\right)\,dx\\
\geq&\left(k_{1}F(T)-\frac{1}{2}V_{2}T^{2}\right)\big|B(0,R)\big|-\big|B(0,R+1)\setminus B(0,R)\big|\cdot\max_{t\in[0,T]}\big|k_{1} F(t)-\frac{1}{2}V_{2}t^{2}\big|\\
\geq&C_{1}R^{N}-C_{2}R^{N-1},
\end{split}
\end{equation*}
where $|\cdot|$ denotes the Lebesgue measure and $C_{1}$, $C_{2}$ are positive constants. So we could choose $R>0$ large enough such that
$$\int_{\mathbb{R}^{N}}\left(k_{1} F(w)-\frac{1}{2}V_{2}w^{2}\right)\,dx>0.$$
Define
$$\overline{\eta}=\max\left\{\frac{1}{2},\frac{\int_{\mathbb{R}^{N}}V_{2}w^2\,dx}{\int_{\mathbb{R}^{N}}k_{1}F(w)\,dx}\right\},$$
then we have that $\overline{\eta}\geq 1/2$.
Thus, for any $\eta\in[\overline{\eta},1]$ and $\theta>0$, from (K1) it follows that
 \begin{equation*}
\begin{split}
I_{\eta}(w(\frac{x}{\theta}))
\leq&\frac{1}{2}\int_{\mathbb{R}^{N}}\left(|(-\Delta)^{\frac{\alpha}{2}}w(\frac{x}{\theta})|^{2}+V(x)|w(\frac{x}{\theta})|^{2}\right)\,dx
-\eta\int_{\mathbb{R}^{N}}k(x)F(w(\frac{x}{\theta}))\,dx\\
\leq&\frac{1}{2}\theta^{N-2\alpha}\int_{\mathbb{R}^{N}}|(-\Delta)^{\frac{\alpha}{2}}w|^{2}\,dx
+\frac{1}{2}\theta^{N}V_{2}\int_{\mathbb{R}^{N}}w^{2}\,dx
-\theta^{N}\overline{\eta}\int_{\mathbb{R}^{N}}k_{1}F(w)\,dx\\
=&\frac{1}{2}\theta^{N-2\alpha}\int_{\mathbb{R}^{N}}|(-\Delta)^{\frac{\alpha}{2}}w|^{2}\,dx-
\frac{1}{2}\theta^{N}\max\left\{\int_{\mathbb{R}^{N}}V_{2}w^2\,dx,\frac{1}{2}\int_{\mathbb{R}^{N}}k_{1}F(w)\,dx\right\}.
\end{split}
\end{equation*}
Then there exists $\overline{\theta}>0$ such that for any $\theta\geq\overline{\theta}$, $I_{\eta}(w(x/\theta))<0$. We take $v_{0}(x)=w(x/\overline{\theta})$. Therefore the proof is complete.
\end{proof}

\begin{lemma}\label{Le3.2}
 For any $\eta\in[\overline{\eta},1]$, define $$c_{\eta}=\inf_{\gamma\in\Gamma_{\eta}}\max_{t\in[0,1]}I_{\eta}(\gamma(t)),$$
 where $\Gamma_{\eta}=\{\gamma\in C([0,1],H^{\alpha}_{r}(\mathbb{R}^{N})):\gamma(0)=0,\gamma(1)=v_{0}\}$, $\overline{\eta}$ and $v_{0}$ are from \textup{Lemma 3.1}. Then $c_{\eta}>\max\{I_{\eta}(0),I_{\eta}(v_{0})\}$  and there exists $c_{0}>0$  such that $c_{\eta}\leq c_{0}$ for any $\eta\in[\overline{\eta},1]$, where $c_{0}$ is independent of $\lambda$.
\end{lemma}

\begin{proof} According to (H1) and (H2),  for any $\varepsilon>0$, there exists a constant $C(\varepsilon)>0$ such that for any $t\in\mathbb{R}$,
\begin{eqnarray}\label{1}
\begin{split}
f(t)\leq \varepsilon |t|+C(\varepsilon)|t|^{2_{\alpha}^{*}-1}.
\end{split}
\end{eqnarray}
By (\ref{1}), for any $\varepsilon\in(0,1)$, we get
 \begin{equation}\label{2}
F(t)\leq\varepsilon t^2+C(\varepsilon)|t|^{2_{\alpha}^{*}}.
 \end{equation}
Taking $\varepsilon=V_{1}/(4k_{2})$, for any $u\in H^{\alpha}_{r}(\mathbb{R}^{N})$ and $\eta\in[\overline{\eta},1]$, we obtain
 \begin{equation*}
\begin{split}
I_{\eta}(u)\geq&\frac{1}{2}\int_{\mathbb{R}^{N}}(|(-\Delta)^{\frac{\alpha}{2}} u|^{2}+V_{1}u^2)\,dx
-\int_{\mathbb{R}^{N}}k_{2}F(u)\,dx-\frac{\lambda}{2_{\alpha}^{*}}\int_{\mathbb{R}^{N}}|u|^{2_{\alpha}^{*}}\,dx\\
\geq&\frac{1}{2}\int_{\mathbb{R}^{N}}|(-\Delta)^{\frac{\alpha}{2}} u|^{2}\,dx
+\left(V_{1}/2-\varepsilon k_{2}\right)\int_{\mathbb{R}^{N}}u^2\,dx-C(\varepsilon)k_{2}\int_{\mathbb{R}^{N}}|u|^{2_{\alpha}^{*}}\,dx
-\frac{\lambda}{2_{\alpha}^{*}}\int_{\mathbb{R}^{N}}|u|^{2_{\alpha}^{*}}\,dx\\
\geq&\frac{1}{2}\int_{\mathbb{R}^{N}}|(-\Delta)^{\frac{\alpha}{2}} u|^{2}\,dx
+\frac{V_{1}}{4}\int_{\mathbb{R}^{N}}u^2\,dx-C\int_{\mathbb{R}^{N}}|u|^{2_{\alpha}^{*}}\,dx
-\frac{\lambda}{2_{\alpha}^{*}}\int_{\mathbb{R}^{N}}|u|^{2_{\alpha}^{*}}\,dx\\
\geq&\min\left\{1/2,V_{1}/4\right\}||u||_{H^{\alpha}(\mathbb{R}^{N})}^{2}
-C||u||_{H^{\alpha}(\mathbb{R}^{N})}^{2_{\alpha}^{*}}.
\end{split}
\end{equation*}
Thanks to $2_{\alpha}^{*}>2$, there exist $0<\rho<||v_{0}||_{H^{\alpha}(\mathbb{R}^{N})}$ and $\sigma>0$ such that $I_{\eta}(u)\geq\sigma$ for any $u\in H^{\alpha}_{r}(\mathbb{R}^{N})$ with $||u||_{H^{\alpha}(\mathbb{R}^{N})}=\rho$. For any  $\gamma\in\Gamma_{\eta}$, we have $\gamma(0)=0$ and $\gamma(1)=v_0$. Then, there exists $t_{\eta}\in(0,1)$ such that $||\gamma(t_{\eta})||_{H^{\alpha}(\mathbb{R}^{N})}=\rho$, which implies
$$c_{\eta}\geq\inf_{\gamma\in\Gamma_{\eta}}I_{\eta}(\gamma(t_{\eta}))\geq\sigma>\max\{I_{\eta}(0),I_{\eta}(v_{0})\}.$$

Take $\gamma_{0}(t)=tv_{0}$, then $\gamma_{0}\in\Gamma_{\eta}$. For any $t\in[0,1]$,  we obtain
 \begin{equation*}
\begin{split}
I_{\eta}(\gamma_{0}(t))=I_{\eta}(tv_{0})
\leq&\frac{1}{2}\int_{\mathbb{R}^{N}}(|(-\Delta)^{\frac{\alpha}{2}}v_{0}|^{2}+V(x)v_{0}^{2})\,dx\triangleq c_{0},
\end{split}
\end{equation*}
which implies that
$c_{\eta}\leq\max_{t\in[0,1]}I_{\eta}(\gamma_{0}(t))\leq c_{0}$ for any $\eta\in[\overline{\eta},1]$.
 Thus we have completed the proof.
 \end{proof}

\begin{theorem}\label{th3.1} {\em (\cite[Theorem 1.1]{Jeanjean1})} Let $(X,||\cdot||_{X})$ be a Banach space and $I\subset\mathbb{R}^{+}$ an interval. Consider a family $\{J_{\eta}\}_{\eta\in I}$ of $C^1$ functionals on $X$ with the form
 $$J_{\eta}(u)=A(u)-\eta B(u),\ \forall\eta\in I,$$
 where $B(u)\geq0$, $\forall u\in X$, and such that either $A(u)\rightarrow+\infty$ or $B(u)\rightarrow+\infty$ as $||u||_{X}\rightarrow\infty$.
  If there are two points $v_{1}$, $v_2\in X$ such that
   $$c_{\eta}=\inf_{\gamma\in\Gamma_{\eta}}\max_{t\in[0,1]}J_{\eta}(\gamma(t))>\max\{J_{v_1},J_{v_2}\},\ \eta\in I,$$
   where
   $$\Gamma_{\eta}=\{\gamma\in C([0,1],X):\gamma(0)=v_1,\gamma(1)=v_{2}\},$$
   then, for almost every $\eta\in I$, there exists a sequence $\{v_n\}\subset X$ such that

   (i) $\{v_n\}$ is bounded;

   (ii) $J_{\eta}(v_n)\rightarrow c_{\eta}$;

   (iii)  $J_{\eta}'(v_n)\rightarrow 0$ in the dual $X'$ of $X$.
  \end{theorem}

  \noindent\textbf{Remark 3.1} In fact, the map $\eta\rightarrow c_{\eta}$ is nonincreasing and continuous from the left (see \cite{Jeanjean1}).

\vspace{2mm}

  By using Lemma 3.1, Lemma 3.2 and  Theorem 3.1,  we obtain that for any $\eta\in[\overline{\eta},1]$, $I_{\eta}$ possesses a bounded (PS) sequence at the level $c_{\eta}$.

Next we will verify that each bounded (PS) sequence for the functional $I_{\eta}$ contains a convergent subsequence.
The main difficulties here are that   the embedding $H_{r}^{\alpha}(\mathbb{R}^{N})\hookrightarrow L^{2^{*}_{\alpha}}(\mathbb{R}^{N})$ is not compact and we do not have a  similar radial lemma (see \cite{Berestycki1983}) in  $H^{\alpha}_{r}(\mathbb{R}^{N})$. To get the compactness of bounded (PS) sequence in $H^{\alpha}_{r}(\mathbb{R}^{N})$, we assume that $\lambda$ in (\ref{p}) is small. Based on the following principle of concentration compactness in $H^{\alpha}_{r}(\mathbb{R}^{N})$ and Lemma 2.4 in \cite{Chang2013}, we obtain Lemma 3.5.

\begin{theorem}\label{th3.2} {\em (\cite[Theorem 1.5]{Palatucci2014})} Let $\Omega\subseteq\mathbb{R}^{N}$ an open subset and let $\{u_n\}$ be a sequence in $\dot{H}^\alpha(\Omega)$ weakly converging to $u$ as $n\rightarrow\infty$ and such that
$$|(-\bigtriangleup)^{\frac{\alpha}{2}}u_{n}|^{2}\rightarrow\mu\  \mbox{and}\ |u_{n}|^{2_{\alpha}^{*}}\rightarrow\nu\  \mbox{weakly-}\ast\  \mbox{in}\  \mathcal{M}(\mathbb{R}^{N}).$$
Then, either $u_{n}\rightarrow u$ in $L_{loc}^{2_{\alpha}^{*}}(\mathbb{R}^{N})$ or there exists a (at most countable) set of distinct points $\{x_j\}_{j\in J}\subset\overline{\Omega}$  and positive numbers $\{\nu_{j}\}_{j\in J}$ such that  we have
$$\nu=|u|^{2_{\alpha}^{*}}+\sum_{j\in J}\nu_{j}\delta_{x_{j}}.$$
If, in addition, $\Omega$ is bounded, then there exist a positive measure $\widetilde{\mu}\in \mathcal{M}(\mathbb{R}^{N})$ with $supp\widetilde{\mu}\subset\overline{\Omega}$ and positive numbers $\{\mu_{j}\}_{j\in J}$ such that
$$\mu=|(-\bigtriangleup)^{\frac{\alpha}{2}}u|^{2}+\widetilde{\mu}+\sum_{j\in J}\mu_{j}\delta_{x_{j}}.$$
\end{theorem}

\noindent\textbf{Remark 3.2} In the case  $\Omega=\mathbb{R}^{N}$, the above principle of concentration compactness  does not provide any information about the possible loss of mass
at infinity. The following result expresses
this fact in quantitative terms, and the proof.

\begin{lemma}\label{Le3.3}  Let $\{u_{n}\}\subset
\dot{H}^{\alpha}(\mathbb{R}^{N})$  such that $u_{n}\rightarrow u$ weakly in $\dot{H}^{\alpha}(\mathbb{R}^{N})$, $|(-\bigtriangleup)^{\frac{\alpha}{2}}u_{n}|^{2}\rightarrow\mu$ and $|u_{n}|^{2_{\alpha}^{*}}\rightarrow\nu$ weakly-$*$ in
$\mathcal{M}(\mathbb{R}^{N})$, as $n\rightarrow\infty$ and define
$$\mu_{\infty}=\lim_{R\rightarrow\infty}\limsup_{n\rightarrow\infty}\int_{\{x\in\mathbb{R}^{N}:|x|>R\}}|(-\Delta)^{\frac{\alpha}{2}}u_{n}|^{2}\,dx,$$
$$\nu_{\infty}=\lim_{R\rightarrow\infty}\limsup_{n\rightarrow\infty}\int_{\{x\in\mathbb{R}^{N}:|x|>R\}}|u_{n}|^{2_{\alpha}^{*}}\,dx.$$
The quantities $\mu_{\infty}$ and $\nu_{\infty}$ are well defined
and satisfy
\begin{equation*}
\limsup_{n\rightarrow\infty}\int_{\mathbb{R}^{N}}|(-\Delta)^{\frac{\alpha}{2}}u_{n}|^{2}\,dx=\int_{\mathbb{R}^{N}}\,d\mu+\mu_{\infty},
\end{equation*}
\begin{equation}\label{14}
\limsup_{n\rightarrow\infty}\int_{\mathbb{R}^{N}}|u_{n}|^{2_{\alpha}^{*}}\,dx=\int_{\mathbb{R}^{N}}\,d\nu+\nu_{\infty}.
\end{equation}
\end{lemma}

\begin{proof} The proof is similar to that of Lemma 3.5 in \cite{ZZX}.
Thus we just give a sketch of the proof for the reader's convenience.
Take  $\phi\in C^{\infty}(\mathbb{R}^{N})$ such that
$0\leq\phi\leq1$; $\phi\equiv1$ in $\mathbb{R}^{N}\setminus
B(0,2)$, $\phi\equiv0$ in $B(0,1)$. For any $R>0$, define
$\phi_{R}(x)=\phi(x/R)$. Then we have
\begin{displaymath}
\begin{split}
\int_{\{x\in\mathbb{R}^{N}:|x|>2R\}}|(-\Delta)^{\frac{\alpha}{2}}u_{n}|^{2}\,dx
\leq\int_{\mathbb{R}^{N}}|(-\Delta)^{\frac{\alpha}{2}}u_{n}|^{2}\phi_{R}\,dx
\leq\int_{\{x\in\mathbb{R}^{N}:|x|>R\}}|(-\Delta)^{\frac{\alpha}{2}}u_{n}|^{2}\,dx,
\end{split}
\end{displaymath}
thus
$\mu_{\infty}=\lim_{R\rightarrow\infty}\limsup_{n\rightarrow\infty}\int_{\mathbb{R}^{N}}|(-\Delta)^{\frac{\alpha}{2}}u_{n}|^{2}\phi_{R}\,dx.$
Note that
\begin{displaymath}
\begin{split}
\int_{\mathbb{R}^{N}}|(-\Delta)^{\frac{\alpha}{2}}u_{n}|^{2}\,dx
=\int_{\mathbb{R}^{N}}|(-\Delta)^{\frac{\alpha}{2}}u_{n}|^{2}\phi_{R}\,dx+\int_{\mathbb{R}^{N}}|(-\Delta)^{\frac{\alpha}{2}}u_{n}|^{2}(1-\phi_{R})\,dx.
\end{split}
\end{displaymath}
It is easy to verify that  $$\int_{\mathbb{R}^{N}}|(-\Delta)^{\frac{\alpha}{2}}u_{n}|^{2}(1-\phi_{R})\,dx\rightarrow\int_{\mathbb{R}^{N}}(1-\phi_{R})\,d\mu,$$
as $n\rightarrow\infty$. Hence we have
$$\mu(\mathbb{R}^{N})=\lim_{R\rightarrow\infty}\lim_{n\rightarrow\infty}\int_{\mathbb{R}^{N}}|(-\Delta)^{\frac{\alpha}{2}}u_{n}|^{2}(1-\phi_{R})\,dx.$$
Then
\begin{displaymath}
\begin{split}
\limsup_{n\rightarrow\infty}\int_{\mathbb{R}^{N}}|(-\Delta)^{\frac{\alpha}{2}}u_{n}|^{2}\,dx
=&\lim_{R\rightarrow\infty}\left(\limsup_{n\rightarrow\infty}\int_{\mathbb{R}^{N}}|(-\Delta)^{\frac{\alpha}{2}}u_{n}|^{2}\phi_{R}\,dx
+\int_{\mathbb{R}^{N}}(1-\phi_{R})\,d\mu\right)\\
=&\mu_{\infty}+\mu(\mathbb{R}^{N}).
\end{split}
\end{displaymath}
Similarly, we obtain that
$\limsup_{n\rightarrow\infty}\int_{\mathbb{R}^{N}}|u_{n}|^{2_{\alpha}^{*}}\,dx=\nu(\mathbb{R}^{N})+\nu_{\infty}.$
The lemma is thus proved. \end{proof}

In the sequel,  we derive some results involving  $\nu_{i}$ for any $i\in J$ and $\nu_{\infty}$.

\begin{lemma}\label{Le3.4}
 Let $\{u_{n}\}\subset
\dot{H}^{\alpha}(\mathbb{R}^{N})$  such that $u_{n}\rightarrow u$ weakly in $\dot{H}^{\alpha}(\mathbb{R}^{N})$, $|(-\bigtriangleup)^{\alpha/2}u_{n}|^{2}\rightarrow\mu$ and $|u_{n}|^{2_{\alpha}^{*}}\rightarrow\nu$ weakly-$*$ in
$\mathcal{M}(\mathbb{R}^{N})$, as $n\rightarrow\infty$. Then,
 $\nu_{i}\leq(S_{\alpha}^{-1}\mu(\{x_{i}\}))^{2_{\alpha}^{*}/2}$ for any $i\in J$ and $\nu_{\infty}\leq(S_{\alpha}^{-1}\mu_{\infty})^{2_{\alpha}^{*}/2}$, where $x_{i}$, $\nu_{i}$ are from \textup{Theorem 3.2} and $\mu_{\infty}$, $\nu_{\infty}$ are from \textup{Lemma 3.3}, $S_{\alpha}$ is the best Sobolev constant of the embedding $\dot{H}^{\alpha}(\mathbb{R}^{N})\hookrightarrow L^{2_{\alpha}^{*}}(\mathbb{R}^{N})$ \textup{(see \cite{Nezza})}, i.e.
\begin{eqnarray}\label{5}
S_{\alpha}=\inf_{u\in\dot{H}^{\alpha}(\mathbb{R}^{N})}
\frac{\int_{\mathbb{R}^{N}}|(-\Delta)^{\frac{\alpha}{2}}u|^{2}\,dx}{||u||_{L^{2_{\alpha}^{*}}(\mathbb{R}^{N})}^{2}}.
\end{eqnarray}
\end{lemma}

\begin{proof} (1) Take $\varphi\in C_{0}^{\infty}(\mathbb{R}^{N})$
such that $0\leq\varphi\leq1$;
$\varphi\equiv1$ in $B(0,1)$, $\varphi\equiv0$ in $\mathbb{R}^{N}\setminus B(0,2)$. For any $\varepsilon>0$, define $\varphi_{\varepsilon}(x)=\varphi(\frac{x-x_{i}}{\varepsilon})$, where $i\in J$. It follows from  (\ref{30}) and (\ref{5}) that
$$\int_{\mathbb{R}^{N}}|u_{n}\varphi_{\varepsilon}|^{2_{\alpha}^{*}}\,dx\leq \left(S_{\alpha}^{-1}\iint_{\mathbb{R}^{2N}}
\frac{|u_{n}(x)\varphi_{\varepsilon}(x)-u_{n}(y)\varphi_{\varepsilon}(y)|^{2}}{|x-y|^{N+2\alpha}}\,dxdy\right)^{2_{\alpha}^{*}/2}.$$
We have
$$\int_{\mathbb{R}^{N}}|u_{n}\varphi_{\varepsilon}|^{2_{\alpha}^{*}}\,dx\rightarrow
\int_{\mathbb{R}^{N}}\varphi_{\varepsilon}^{2_{\alpha}^{*}}\,d\nu,\ \mbox{as}\ n\rightarrow\infty,$$
 $$\int_{\mathbb{R}^{N}}\varphi_{\varepsilon}^{2_{\alpha}^{*}}\,d\nu\rightarrow\nu(\{x_{i}\})=\nu_{i},\ \mbox{as}\ \varepsilon\rightarrow0.$$
Note that
\begin{displaymath}
\begin{split}
&\iint_{\mathbb{R}^{2N}}\frac{|u_{n}(x)\varphi_{\varepsilon}(x)-u_{n}(y)\varphi_{\varepsilon}(y)|^{2}}{|x-y|^{N+2\alpha}}\,dxdy\\
=&\iint_{\mathbb{R}^{2N}}\frac{|u_{n}(x)\varphi_{\varepsilon}(x)-u_{n}(x)\varphi_{\varepsilon}(y)
+u_{n}(x)\varphi_{\varepsilon}(y)-u_{n}(y)\varphi_{\varepsilon}(y)|^{2}}{|x-y|^{N+2\alpha}}\,dxdy\\
=&\iint_{\mathbb{R}^{2N}}\frac{u_{n}^{2}(x)(\varphi_{\varepsilon}(x)-\varphi_{\varepsilon}(y))^{2}}{|x-y|^{N+2\alpha}}\,dxdy
+\iint_{\mathbb{R}^{2N}}\frac{\varphi_{\varepsilon}^{2}(y)(u_{n}(x)-u_{n}(y))^{2}}{|x-y|^{N+2\alpha}}\,dxdy\\
&+\iint_{\mathbb{R}^{2N}}\frac{2u_{n}(x)\varphi_{\varepsilon}(y)(u_{n}(x)-u_{n}(y))(\varphi_{\varepsilon}(x)-\varphi_{\varepsilon}(y))}{|x-y|^{N+2\alpha}}\,dxdy,
\end{split}
\end{displaymath}
we get
$$\iint_{\mathbb{R}^{2N}}\frac{\varphi_{\varepsilon}^{2}(y)(u_{n}(x)-u_{n}(y))^{2}}{|x-y|^{N+2\alpha}}\,dxdy
\rightarrow\int_{\mathbb{R}^{N}}\varphi_{\varepsilon}^{2}\,d\mu,\ \mbox{as}\  n\rightarrow\infty,$$
$$\int_{\mathbb{R}^{N}}\varphi_{\varepsilon}^{2}\,d\mu\rightarrow\mu(\{x_{i}\}),\ \mbox{as}\ \varepsilon\rightarrow0.$$
Since $\{u_n\}$ is bounded in $\dot{H}^{\alpha}(\mathbb{R}^{N})$, by the H\"{o}lder inequality we obtain
\begin{displaymath}
\begin{split}
&\left|\iint_{\mathbb{R}^{2N}}\frac{u_{n}(x)\varphi_{\varepsilon}(y)(u_{n}(x)-u_{n}(y))(\varphi_{\varepsilon}(x)-\varphi_{\varepsilon}(y))}{|x-y|^{N+2\alpha}}\,dxdy\right|\\
\leq&\left(\iint_{\mathbb{R}^{2N}}\frac{\varphi_{\varepsilon}^{2}(y)(u_{n}(x)-u_{n}(y))^{2}}{|x-y|^{N+2\alpha}}\,dxdy\right)^{\frac{1}{2}}\left(\iint_{\mathbb{R}^{2N}}\frac{u_{n}^{2}(x)(\varphi_{\varepsilon}(x)-\varphi_{\varepsilon}(y))^{2}}
{|x-y|^{N+2\alpha}}\,dxdy\right)^{1/2}\\
\leq&C\left(\iint_{\mathbb{R}^{2N}}\frac{u_{n}^{2}(x)(\varphi_{\varepsilon}(x)-\varphi_{\varepsilon}(y))^{2}}
{|x-y|^{N+2\alpha}}\,dxdy\right)^{1/2}.
\end{split}
\end{displaymath}

In the following, we claim that $$\lim_{\varepsilon\rightarrow0}\lim_{n\rightarrow\infty}\iint_{\mathbb{R}^{2N}}\frac{u_{n}^{2}(x)(\varphi_{\varepsilon}(x)-\varphi_{\varepsilon}(y))^{2}}
{|x-y|^{N+2\alpha}}\,dxdy=0.$$
Note that
\begin{displaymath}
\begin{split}
\mathbb{R}^{N}\times\mathbb{R}^{N}
=&((\mathbb{R}^{N}\setminus B(x_{i},2\varepsilon))\cup B(x_{i},2\varepsilon))
\times((\mathbb{R}^{N}\setminus B(x_{i},2\varepsilon))\cup B(x_{i},2\varepsilon))\\
=&((\mathbb{R}^{N}\setminus B(x_{i},2\varepsilon))\times(\mathbb{R}^{N}\setminus B(x_{i},2\varepsilon)))\cup(B(x_{i},2\varepsilon)\times\mathbb{R}^{N})\\
&\cup((\mathbb{R}^{N}\setminus B(x_{i},2\varepsilon))\times B(x_{i},2\varepsilon)).
\end{split}
\end{displaymath}

(i) If $(x,y)\in(\mathbb{R}^{N}\setminus B(x_{i},2\varepsilon))\times(\mathbb{R}^{N}\setminus B(x_{i},2\varepsilon))$, then $\varphi_{\varepsilon}(x)=\varphi_{\varepsilon}(y)=0$.

(ii) $(x,y)\in B(x_{i},2\varepsilon)\times\mathbb{R}^{N}$. If $|x-y|\leq\varepsilon$,
$|y-x_{i}|\leq|x-y|+|x-x_{i}|\leq3\varepsilon,$ which implies
\begin{displaymath}
\begin{split}
&\int_{B(x_{i},2\varepsilon)}dx\int_{\{y\in\mathbb{R}^{N}:|x-y|\leq\varepsilon\}}\frac{u_{n}^{2}(x)(\varphi_{\varepsilon}(x)-\varphi_{\varepsilon}(y))^{2}}{|x-y|^{N+2\alpha}}dy\\
=&\int_{B(x_{i},2\varepsilon)}dx\int_{\{y\in\mathbb{R}^{N}:|x-y|\leq\varepsilon\}}\frac{u_{n}^{2}(x)|\nabla\varphi(\xi)|^{2}|\frac{x-y}{\varepsilon}|^{2}}{|x-y|^{N+2\alpha}}dy\\
\leq& C\varepsilon^{-2}
\int_{B(x_{i},2\varepsilon)}dx\int_{\{y\in\mathbb{R}^{N}:|x-y|\leq\varepsilon\}}\frac{u_{n}^{2}(x)}{|x-y|^{N+2\alpha-2}}dy\\
=&C\varepsilon^{-2\alpha}
\int_{B(x_{i},2\varepsilon)}u_{n}^{2}(x)dx,
\end{split}
\end{displaymath}
where $\xi=(y-x_{i})/\varepsilon+\tau(x-x_{i})/\varepsilon$ and $\tau\in(0,1)$.

If $|x-y|>\varepsilon$, then we have
\begin{displaymath}
\begin{split}
&\int_{B(x_{i},2\varepsilon)}dx\int_{\{y\in\mathbb{R}^{N}:|x-y|>\varepsilon\}}\frac{u_{n}^{2}(x)(\varphi_{\varepsilon}(x)-\varphi_{\varepsilon}(y))^{2}}{|x-y|^{N+2\alpha}}dy\\
\leq& C
\int_{B(x_{i},2\varepsilon)}dx\int_{\{y\in\mathbb{R}^{N}:|x-y|>\varepsilon\}}\frac{u_{n}^{2}(x)}{|x-y|^{N+2\alpha}}dy\\
=&C\varepsilon^{-2\alpha}
\int_{B(x_{i},2\varepsilon)}u_{n}^{2}(x)dx.
\end{split}
\end{displaymath}

(iii) $(x,y)\in(\mathbb{R}^{N}\setminus B(x_{i},2\varepsilon))\times B(x_{i},2\varepsilon)$.
If $|x-y|\leq\varepsilon$, $|x-x_{i}|\leq|x-y|+|y-x_{i}|\leq3\varepsilon.$ Then
\begin{displaymath}
\begin{split}
&\int_{\mathbb{R}^{N}\setminus B(x_{i},2\varepsilon)}dx\int_{\{y\in B(x_{i},2\varepsilon):|x-y|\leq\varepsilon\}}\frac{u_{n}^{2}(x)(\varphi_{\varepsilon}(x)-\varphi_{\varepsilon}(y))^{2}}{|x-y|^{N+2\alpha}}dy\\
\leq& C\varepsilon^{-2}
\int_{B(x_{i},3\varepsilon)}dx\int_{\{y\in B(x_{i},2\varepsilon):|x-y|\leq\varepsilon\}}\frac{u_{n}^{2}(x)}{|x-y|^{N+2\alpha-2}}dy\\
\leq&C\varepsilon^{-2}
\int_{B(x_{i},3\varepsilon)}dx\int_{\{z\in\mathbb{R}^{N}:|z|\leq\varepsilon\}}\frac{u_{n}^{2}(x)}{|z|^{N+2\alpha-2}}dz\\
=&C\varepsilon^{-2\alpha}
\int_{B(x_{i},3\varepsilon)}u_{n}^{2}(x)dx.
\end{split}
\end{displaymath}

Notice that there exists $K>4$ such that
$(\mathbb{R}^{N}\setminus B(x_{i},2\varepsilon))\times B(x_{i},2\varepsilon)
\subset(B(x_{i},K\varepsilon)\times B(x_{i},2\varepsilon))\cup((\mathbb{R}^{N}\setminus B(x_{i},K\varepsilon))\times B(x_{i},2\varepsilon)).$

If $|x-y|>\varepsilon$, we obtain
\begin{displaymath}
\begin{split}
&\int_{ B(x_{i},K\varepsilon)}dx\int_{\{y\in B(x_{i},2\varepsilon):|x-y|>\varepsilon\}}\frac{u_{n}^{2}(x)(\varphi_{\varepsilon}(x)-\varphi_{\varepsilon}(y))^{2}}{|x-y|^{N+2\alpha}}dy\\
\leq& C
\int_{B(x_{i},K\varepsilon)}dx\int_{\{y\in B(x_{i},2\varepsilon):|x-y|>\varepsilon\}}\frac{u_{n}^{2}(x)}{|x-y|^{N+2\alpha}}dy\\
\leq&C
\int_{B(x_{i},K\varepsilon)}dx\int_{\{z\in\mathbb{R}^{N}:|z|>\varepsilon\}}\frac{u_{n}^{2}(x)}{|z|^{N+2\alpha}}dz\\
\leq&C\varepsilon^{-2\alpha}
\int_{B(x_{i},K\varepsilon)}u_{n}^{2}(x)dx.
\end{split}
\end{displaymath}

If $(x,y)\in(\mathbb{R}^{N}\setminus B(x_{i},K\varepsilon))\times B(x_{i},2\varepsilon)$, we get
\begin{displaymath}
\begin{split}
|x-y|&\geq|x-x_{i}|-|y-x_{i}|=\frac{|x-x_{i}|}{2}+\frac{|x-x_{i}|}{2}-|y-x_{i}|\\
&\geq\frac{|x-x_{i}|}{2}+\frac{K}{2}\varepsilon-2\varepsilon>\frac{|x-x_{i}|}{2},
\end{split}
\end{displaymath}
which implies
\begin{displaymath}
\begin{split}
&\int_{\mathbb{R}^{N}\setminus B(x_{i},K\varepsilon)}dx\int_{\{y\in B(x_{i},2\varepsilon):|x-y|>\varepsilon\}}\frac{u_{n}^{2}(x)(\varphi_{\varepsilon}(x)-\varphi_{\varepsilon}(y))^{2}}{|x-y|^{N+2\alpha}}dy\\
\leq&C
\int_{\mathbb{R}^{N}\setminus B(x_{i},K\varepsilon)}dx\int_{\{y\in B(x_{i},2\varepsilon):|x-y|>\varepsilon\}}\frac{u_{n}^{2}(x)}{|x-x_{i}|^{N+2\alpha}}dy\\
\leq&C\varepsilon^{N}
\int_{\mathbb{R}^{N}\setminus B(x_{i},K\varepsilon)}\frac{u_{n}^{2}(x)}{|x-x_{i}|^{N+2\alpha}}dx\\
\leq&C\varepsilon^{N}
\left(\int_{\mathbb{R}^{N}\setminus B(x_{i},K\varepsilon)}|u_{n}(x)|^{2_{\alpha}^{*}}\,dx\right)^{2/2_{\alpha}^{*}}
\left(\int_{\mathbb{R}^{N}\setminus B(x_{i},K\varepsilon)}|x-x_{i}|^{-(N+2\alpha)\frac{2_{\alpha}^{*}}{2_{\alpha}^{*}-2}}\,dx\right)^{(2_{\alpha}^{*}-2)/2_{\alpha}^{*}}\\
=&CK^{-N}\left(\int_{\mathbb{R}^{N}\setminus B(x_{i},K\varepsilon)}|u_{n}(x)|^{2_{\alpha}^{*}}\,dx\right)^{2/2_{\alpha}^{*}}.
\end{split}
\end{displaymath}
In views of (i), (ii) and (iii), we have
\begin{eqnarray}\label{4}
\begin{split}
&\iint_{\mathbb{R}^{2N}}\frac{u_{n}^{2}(x)(\varphi_{\varepsilon}(x)-\varphi_{\varepsilon}(y))^{2}}{|x-y|^{N+2\alpha}}\,dxdy\\
=&\iint_{B(x_{i},2\varepsilon)\times\mathbb{R}^{N}}\frac{u_{n}^{2}(x)(\varphi_{\varepsilon}(x)-\varphi_{\varepsilon}(y))^{2}}{|x-y|^{N+2\alpha}}\,dxdy +\iint_{(\mathbb{R}^{N}\setminus B(x_{i},2\varepsilon))\times B(x_{i},2\varepsilon)} \frac{u_{n}^{2}(x)(\varphi_{\varepsilon}(x)-\varphi_{\varepsilon}(y))^{2}}{|x-y|^{N+2\alpha}}\,dxdy\\
\leq&C\varepsilon^{-2\alpha}\int_{B(x_{i},K\varepsilon)}u_{n}^{2}(x)\,dx+CK^{-N}\left(\int_{\mathbb{R}^{N}\setminus B(x_{i},K\varepsilon)}|u_{n}(x)|^{2_{\alpha}^{*}}\,dx\right)^{2/2_{\alpha}^{*}}\\
\leq&C\varepsilon^{-2\alpha}\int_{B(x_{i},K\varepsilon)}u_{n}^{2}(x)\,dx+CK^{-N}.
\end{split}
\end{eqnarray}
Note that $u_{n}\rightarrow u$ weakly in $H^{\alpha}(\mathbb{R}^{N})$,
by Theorem 2.1 we obtain
  $u_{n}\rightarrow u$ in $L^{2}_{loc}(\mathbb{R}^{N})$, which implies
$$C\varepsilon^{-2\alpha}\int_{B(x_{i},K\varepsilon)}u_{n}^{2}(x)\,dx
+CK^{-N}\rightarrow C\varepsilon^{-2\alpha}\int_{B(x_{i},K\varepsilon)}u^{2}(x)\,dx+CK^{-N},$$
 as $n\rightarrow\infty$. Then,
\begin{displaymath}
\begin{split}
&C\varepsilon^{-2\alpha}\int_{B(x_{i},K\varepsilon)}u^{2}(x)\,dx+CK^{-N}\\
\leq& C\varepsilon^{-2\alpha}\left(\int_{B(x_{i},K\varepsilon)}|u(x)|^{2_{\alpha}^{*}}\,dx\right)^{2/2_{\alpha}^{*}}
\left(\int_{B(x_{i},K\varepsilon)}\,dx\right)^{1-2/2_{\alpha}^{*}}+CK^{-N}\\
=& CK^{2\alpha}\left(\int_{B(x_{i},K\varepsilon)}|u(x)|^{2_{\alpha}^{*}}\,dx\right)^{2/2_{\alpha}^{*}}+CK^{-N}\rightarrow CK^{-N}
\end{split}
\end{displaymath}
as $\varepsilon\rightarrow0$.
Furthermore, we have
\begin{eqnarray}\label{7}
\begin{split}
&\limsup_{\varepsilon\rightarrow0}\limsup_{n\rightarrow\infty}
\iint_{\mathbb{R}^{2N}}\frac{u_{n}^{2}(x)(\varphi_{\varepsilon}(x)-\varphi_{\varepsilon}(y))^{2}}{|x-y|^{N+2\alpha}}\,dxdy\\
=&\lim_{K\rightarrow\infty}\limsup_{\varepsilon\rightarrow0}\limsup_{n\rightarrow\infty}
\iint_{\mathbb{R}^{2N}}\frac{u_{n}^{2}(x)(\varphi_{\varepsilon}(x)-\varphi_{\varepsilon}(y))^{2}}{|x-y|^{N+2\alpha}}\,dxdy=0.
\end{split}
\end{eqnarray}
Thus, for any $i\in J$,
we obtain $$\nu_{i}\leq(S_{\alpha}^{-1}\mu(\{x_{i}\}))^{2_{\alpha}^{*}/2}.$$

(2) It follows from from  (\ref{30}) and (\ref{5}) that
$$\int_{\mathbb{R}^{N}}|u_{n}\chi_{R}|^{2_{\alpha}^{*}}\,dx\leq
\left(S_{\alpha}^{-1}\iint_{\mathbb{R}^{2N}}\frac{|u_{n}(x)\chi_{R}(x)-u_{n}(y)\chi_{R}(y)|^{2}}{|x-y|^{N+2\alpha}}\,dxdy\right)^{2_{\alpha}^{*}/2},$$
where $\chi_{R}$ is from Lemma 3.3. We have
$$\limsup_{R\rightarrow\infty}\limsup_{n\rightarrow\infty}\int_{\mathbb{R}^{N}}|u_{n}\chi_{R}|^{2_{\alpha}^{*}}\,dx=\nu_{\infty}.$$
Note that
\begin{displaymath}
\begin{split}
&\iint_{\mathbb{R}^{2N}}\frac{|u_{n}(x)\chi_{R}(x)-u_{n}(y)\chi_{R}(y)|^{2}}{|x-y|^{N+2\alpha}}\,dxdy\\
=&\iint_{\mathbb{R}^{2N}}\frac{u_{n}^{2}(x)(\chi_{R}(x)-\chi_{R}(y))^{2}}{|x-y|^{N+2\alpha}}\,dxdy
+\iint_{\mathbb{R}^{2N}}\frac{\chi_{R}^{2}(y)(u_{n}(x)-u_{n}(y))^{2}}{|x-y|^{N+2\alpha}}\,dxdy\\
&+\iint_{\mathbb{R}^{2N}}\frac{2u_{n}(x)\chi_{R}(y)(u_{n}(x)-u_{n}(y))(\chi_{R}(x)-\chi_{R}(y))}{|x-y|^{N+2\alpha}}\,dxdy.
\end{split}
\end{displaymath}
We obtain
$$\limsup_{R\rightarrow\infty}\limsup_{n\rightarrow\infty}
\int\int_{\mathbb{R}^{2N}}\frac{\chi_{R}^{2}(y)(u_{n}(x)-u_{n}(y))^{2}}{|x-y|^{N+2\alpha}}\,dxdy=\mu_{\infty}$$
and it follows from the H\"{o}lder inequality that
\begin{displaymath}
\begin{split}
&\left|\iint_{\mathbb{R}^{2N}}\frac{u_{n}(x)\chi_{R}(y)(u_{n}(x)-u_{n}(y))(\chi_{R}(x)-\chi_{R}(y))}{|x-y|^{N+2\alpha}}\,dxdy\right|\\
\leq&C\left(\iint_{\mathbb{R}^{2N}}\frac{u_{n}^{2}(x)(\chi_{R}(x)-\chi_{R}(y))^{2}}
{|x-y|^{N+2\alpha}}\,dxdy\right)^{1/2}.
\end{split}
\end{displaymath}

Note that
\begin{displaymath}
\begin{split}
&\limsup_{R\rightarrow\infty}\limsup_{n\rightarrow\infty}\iint_{\mathbb{R}^{2N}}\frac{u_{n}^{2}(x)(\chi_{R}(x)-\chi_{R}(y))^{2}}
{|x-y|^{N+2\alpha}}\,dxdy\\
=&\limsup_{R\rightarrow\infty}\limsup_{n\rightarrow\infty}\iint_{\mathbb{R}^{2N}}\frac{u_{n}^{2}(x)((1-\chi_{R}(x))-(1-\chi_{R}(y)))^{2}}
{|x-y|^{N+2\alpha}}\,dxdy,
\end{split}
\end{displaymath}
then, similar to the proof of (\ref{7}), we obtain
\begin{displaymath}
\begin{split}
\limsup_{R\rightarrow\infty}\limsup_{n\rightarrow\infty}\iint_{\mathbb{R}^{2N}}\frac{u_{n}^{2}(x)((1-\chi_{R}(x))-(1-\chi_{R}(y)))^{2}}
{|x-y|^{N+2\alpha}}\,dxdy=0.
\end{split}
\end{displaymath}
Then, $$\nu_{\infty}\leq(S_{\alpha}^{-1}\mu_{\infty})^{2_{\alpha}^{*}/2}.$$
Therefore, we have completed the proof. \end{proof}

\begin{lemma}\label{Le3.5}
 There exists $\lambda_*>0$ such that for any $\lambda\in(0,\lambda_*)$ and $\eta\in[\overline{\eta},1]$, each bounded $\mathrm{(PS)}$ sequence for  functional $I_{\eta}$ contains a convergent subsequence.
\end{lemma}

 \begin{proof} Let
 $\{u_{n}\}\subset H^{\alpha}_r(\mathbb{R}^{N})$  be a  bounded $\mathrm{(PS)}$ sequence, i.e. there exists $C_3>0$ such that  $$|I_{\eta}(u_{n})|\leq C_3$$
  and
   $$I_{\eta}'(u_{n})\rightarrow0\ \mbox{in}\  H^{\alpha}_r(\mathbb{R}^{N}),\ \mbox{as}\ n\rightarrow\infty.$$
 Passing to a subsequence, still
denoted by $\{u_{n}\}$, we may assume that $u_{n}\rightarrow u_{0}$ weakly in $H^{\alpha}_{r}(\mathbb{R}^{N})$. By the compact embedding
 $$H^{\alpha}_r(\mathbb{R}^{N})\hookrightarrow L^p(\mathbb{R}^{N})$$
  for $p\in(2,2^*_\alpha)$, we assume that
  $$u_{n}\rightarrow u_{0}\ \mbox{in}\  L^p(\mathbb{R}^{N})\
   \mbox{and}\ u_{n}(x)\rightarrow u_{0}(x)\ \mbox{a.e.in}\  \mathbb{R}^{N},$$
    as $n\rightarrow\infty$. Moreover, by Phrokorov's Theorem  (see Theorem 8.6.2 in \cite{Bog})
 there exist $\mu, \,\nu\in
\mathcal{M}(\mathbb{R}^{N})$ such that
$$|(-\bigtriangleup)^{\frac{\alpha}{2}}u_{n}|^{2}\rightarrow\mu\ \mbox{and}\ |u_{n}|^{2_{\alpha}^{*}}\rightarrow\nu\ \mbox{weakly-}\ast\ \mbox{in}\  \mathcal{M}(\mathbb{R}^{N}),$$
 as $n\rightarrow\infty$. It follows from Theorem 3.2 that
$u_{n}\rightarrow u_{0}$ in $L_{loc}^{2_{\alpha}^{*}}(\mathbb{R}^{N})$ or  $\nu=|u_{0}|^{2_{\alpha}^{*}}+\sum_{j\in J}\nu_{j}\delta_{x_{j}}$,  as $n\rightarrow\infty$, where
$J$ is a countable set,
 $\{\nu_{j}\}\subset[0,\infty)$, $\{x_{j}\}\subset\mathbb{R}^{N}$.

 For any $\phi\in H^{\alpha}_r(\mathbb{R}^{N})$, we obtain
  \begin{eqnarray*}
\begin{split}
\langle I_{\eta}'(u_{n}), \phi\rangle-\langle I_{\eta}'(u_{0}), \phi\rangle
=&\int_{\mathbb{R}^{N}}(-\Delta)^{\frac{\alpha}{2}}(u_{n}-u_{0})(-\Delta)^{\frac{\alpha}{2}}\phi\,dx
+\int_{\mathbb{R}^{N}}V(x)(u_{n}-u_{0})\phi\,dx\\
&-\eta\int_{\mathbb{R}^{N}}k(x)(f(u_n)-f(u_0))\phi\,dx-\eta\lambda\int_{\mathbb{R}^{N}}(|u_n|^{2_{\alpha}^{*}-2}u_n-|u_{0}|^{2_{\alpha}^{*}-2}u_{0})\phi\,dx.
\end{split}
\end{eqnarray*}
 As $u_{n}\rightarrow u_{0}$ weakly in $H^{\alpha}_{r}(\mathbb{R}^{N})$, we have
$$\int_{\mathbb{R}^{N}}(-\Delta)^{\frac{\alpha}{2}}(u_{n}-u_{0})\cdot(-\Delta)^{\frac{\alpha}{2}}\phi\,dx+\int_{\mathbb{R}^{N}}V(x)(u_{n}-u_{0})\phi\,dx\rightarrow0.$$
Note that
\begin{align*}
\left\{|u_n|^{2_{\alpha}^{*}-2}u_n-|u_{0}|^{2_{\alpha}^{*}-2}u_{0}\right\}_n\ \ \mbox{is bounded in} \
L^{\frac{2_{\alpha}^{*}}{2_{\alpha}^{*}-1}}(\mathbb{R}^{N})
\end{align*}
and
\begin{align*}
|u_n(x)|^{2_{\alpha}^{*}-2}u_n(x)-|u_{0}(x)|^{2_{\alpha}^{*}-2}u_{0}(x)\rightarrow0  \mbox{ a.e.\ in}\  \mathbb{R}^{N},
\end{align*}
then
\begin{align*}
|u_n|^{2_{\alpha}^{*}-2}u_n-|u_{0}|^{2_{\alpha}^{*}-2}u_{0}\rightarrow0  \ \ \mbox{ weakly in}\ L^{\frac{2_{\alpha}^{*}}{2_{\alpha}^{*}-1}}(\mathbb{R}^{N})
\end{align*}
which implies
 $$\int_{\mathbb{R}^{N}}(|u_n|^{2_{\alpha}^{*}-2}u_n-|u_{0}|^{2_{\alpha}^{*}-2}u_{0})\phi\,dx\rightarrow0.$$

In the sequel, we will verify that $\int_{\mathbb{R}^{N}}k(x)(f(u_n)-f(u_0))\phi\,dx\rightarrow0$,  as $n\rightarrow\infty$.

 Let $\psi\in C^\infty_0(-2,2)$ such that $\psi\equiv1$ on $(-1,1)$ and define
$f_1(t)=\psi(t)f(t)$, $f_2(t)=(1-\psi(t))f(t)$. Hence we obtain
$$|f_1(t)|\leq C_4|t|\ \mbox{and}\ |f_2(t)|\leq C_5|t|^{2_{\alpha}^{*}-1}.$$
Since $\{k(x)f_1(u_n)\}$ is bounded in $L^{2}(\mathbb{R}^{N})$ and $k(x)f_1(u_n(x))\rightarrow k(x)f_1(u_{0}(x))$ a.e. in $\mathbb{R}^{N}$, we get that $k(x)f_1(u_n)\rightarrow k(x)f_1(u_{0})$ weakly in $L^{2}(\mathbb{R}^{N})$.
Thus $$\int_{\mathbb{R}^{N}}k(x)f_1(u_n)\phi\,dx\rightarrow\int_{\mathbb{R}^{N}}k(x)f_1(u_{0})\phi\,dx.$$
 Similarly, $$\int_{\mathbb{R}^{N}}k(x)f_2(u_n)\phi\,dx\rightarrow\int_{\mathbb{R}^{N}}k(x)f_2(u_{0})\phi\,dx.$$
Note that $f(t)=f_1(t)+f_2(t)$, we deduce
$$\int_{\mathbb{R}^{N}}k(x)f(u_n)\phi\,dx\rightarrow\int_{\mathbb{R}^{N}}k(x)f(u_{0})\phi\,dx.$$

As $\langle I_{\eta}'(u_{n}), \phi\rangle\rightarrow0$, it follows that $\langle I_{\eta}'(u_{0}), \phi\rangle=0$, i.e. $I_{\eta}'(u_{0})=0$. Thus,
\begin{equation}\label{9}
\int_{\mathbb{R}^{N}}|(-\Delta)^{\frac{\alpha}{2}}u_{0}|^2\,dx
+\int_{\mathbb{R}^{N}}V(x)u_0^2\,dx
=\eta\int_{\mathbb{R}^{N}}k(x)f(u_0)u_0\,dx
+\eta\lambda\int_{\mathbb{R}^{N}}|u_{0}|^{2_{\alpha}^{*}}\,dx.
\end{equation}
By Lemma 2.4 in \cite{Chang2013}, we get
\begin{equation}\label{10}
\int_{\mathbb{R}^{N}}k(x)f(u_n)u_n\,dx\rightarrow\int_{\mathbb{R}^{N}}k(x)f(u_0)u_0\,dx,
\end{equation}
 as $n\rightarrow\infty.$
It follows from the Fatou Lemma that
\begin{equation}\label{15}
\int_{\mathbb{R}^{N}}V(x)u_0^2\,dx\leq\liminf_{n\rightarrow\infty}\int_{\mathbb{R}^{N}}V(x)u_n^2\,dx.
\end{equation}

Next we will verify that $u_n\rightarrow u_0$ in $L^{2_{\alpha}^{*}}(\mathbb{R}^{N})$.
To this end, we divide the proof into two steps.

\textbf{Step 1:}\begin{itshape} For any $i\in J$, $\mu(\{x_{i}\})\leq \lambda\nu_{i}$ and $\mu_{\infty}\leq \lambda\nu_{\infty}$.
\end{itshape}

(1) Taking radially symmetric function $\varphi_{\varepsilon}$  as in Lemma 3.4,
we get
\begin{eqnarray}\label{18}
\begin{split}
&\iint_{\mathbb{R}^{2N}}\frac{|u_{n}(x)\varphi_{\varepsilon}(x)-u_{n}(y)\varphi_{\varepsilon}(y)|^{2}}
{|x-y|^{N+2\alpha}}\,dxdy\\
\leq&2\iint_{\mathbb{R}^{2N}}\frac{|u_{n}(x)-u_{n}(y)|^{2}\varphi_{\varepsilon}^{2}(y)}
{|x-y|^{N+2\alpha}}\,dxdy+2\iint_{\mathbb{R}^{2N}}\frac{u_{n}^{2}(x)|\varphi_{\varepsilon}(x)-\varphi_{\varepsilon}(y)|^{2}}
{|x-y|^{N+2\alpha}}\,dxdy\\
\leq&2\iint_{\mathbb{R}^{2N}}\frac{|u_{n}(x)-u_{n}(y)|^{2}}
{|x-y|^{N+2\alpha}}\,dxdy+2\iint_{\mathbb{R}^{2N}}\frac{u_{n}^{2}(x)|\varphi_{\varepsilon}(x)-\varphi_{\varepsilon}(y)|^{2}}
{|x-y|^{N+2\alpha}}\,dxdy.
\end{split}
\end{eqnarray}
Similar to the proof of (\ref{4}), we have
\begin{eqnarray}\label{11}
\begin{split}
\iint_{\mathbb{R}^{2N}}\frac{u_{n}^{2}(x)|\varphi_{\varepsilon}(x)-\varphi_{\varepsilon}(y)|^{2}}
{|x-y|^{N+2\alpha}}\,dxdy
\leq C\varepsilon^{-2\alpha}\int_{B(x_i,K\varepsilon)}u_{n}^{2}(x)\,dx+CK^{-N},
\end{split}
\end{eqnarray}
where $K>4$. As $\{u_{n}\}$ is bounded in $H^{\alpha}_r(\mathbb{R}^{N})$, it follows from (\ref{18}) and (\ref{11}) that $\{u_{n}\varphi_{\varepsilon}\}$ is bounded in $H^{\alpha}_r(\mathbb{R}^{N})$.
Then
$$\langle I_{\eta}'(u_{n}),u_{n}\varphi_{\varepsilon}\rangle\rightarrow0,$$
as $n\rightarrow\infty$,
which implies
\begin{eqnarray}\label{6}
\begin{split}
\int_{\mathbb{R}^{N}}(-\Delta)^{\frac{\alpha}{2}}u_{n}\cdot(-\Delta)^{\frac{\alpha}{2}}(u_{n}\varphi_{\varepsilon})\,dx
=\int_{\mathbb{R}^{N}}\left(\eta k(x) f(u_{n})u_{n}+\eta\lambda|u_{n}|^{2_{\alpha}^{*}}-V(x)u_n^2\right)\varphi_{\varepsilon}\,dx+o(1).
\end{split}
\end{eqnarray}

For any $\tau>0$, by (H2) there exist $p\in(2,2_{\alpha}^{*})$ and $C_6>0$ such that
\begin{equation*}
tf(t)\leq \frac{V_{1}}{2k_{2}}t^{2}+\tau |t|^{2_{\alpha}^{*}}+C_6|t|^{p},
\end{equation*}
which implies
\begin{displaymath}
\begin{split}
&\int_{\mathbb{R}^{N}}\left(\eta k(x) f(u_{n})u_{n}+\eta\lambda|u_{n}|^{2_{\alpha}^{*}}-V(x)u_n^2\right)\varphi_{\varepsilon}\,dx\\
\leq&\int_{\mathbb{R}^{N}}\left(\frac{V_{1}}{2}u_{n}^{2}+\tau k_{2}|u_{n}|^{2_{\alpha}^{*}}
+C_6k_{2}|u_{n}|^{p}+\lambda|u_{n}|^{2_{\alpha}^{*}}-V_{1}u_n^2\right)\varphi_{\varepsilon}\,dx\\
\leq&\int_{\mathbb{R}^{N}}\left(\tau k_{2}|u_{n}|^{2_{\alpha}^{*}}
+C_6 k_{2}|u_{n}|^{p}+\lambda|u_{n}|^{2_{\alpha}^{*}}\right)\varphi_{\varepsilon}\,dx.
\end{split}
\end{displaymath}
Note that
\begin{displaymath}
\begin{split}
\int_{\mathbb{R}^{N}}|u_{n}|^{p}\varphi_{\varepsilon}\,dx
=\int_{B(x_{i},2\varepsilon)}|u_{n}|^{p}\varphi_{\varepsilon}\,dx
\rightarrow\int_{B(x_{i},2\varepsilon)}|u_{0}|^{p}\varphi_{\varepsilon}\,dx,
\end{split}
\end{displaymath}
as $n\rightarrow\infty$ and
$$\int_{B(x_{i},2\varepsilon)}|u_{0}|^{p}\varphi_{\varepsilon}\,dx\rightarrow0,$$
 as $\varepsilon\rightarrow0$, then
\begin{displaymath}
\begin{split}
&\limsup_{\varepsilon\rightarrow0}\limsup_{n\rightarrow\infty}\int_{\mathbb{R}^{N}}\left(\eta k(x) f(u_{n})u_{n}+\eta\lambda|u_{n}|^{2_{\alpha}^{*}}-V(x)u_{n}^2\right)\varphi_{\varepsilon}\,dx\\
\leq&(\tau k_{2}+\lambda)\limsup_{\varepsilon\rightarrow0}\limsup_{n\rightarrow\infty}\int_{\mathbb{R}^{N}}|u_{n}|^{2_{\alpha}^{*}}\varphi_{\varepsilon}\,dx\\
=&(\tau k_{2}+\lambda)\limsup_{\varepsilon\rightarrow0}\int_{\mathbb{R}^{N}}\varphi_{\varepsilon}\,d\nu\\
=&(\tau k_{2}+\lambda)\nu_{i}.
\end{split}
\end{displaymath}
Letting $\tau\rightarrow0$, we get
\begin{eqnarray}\label{8}
\limsup_{\varepsilon\rightarrow0}\limsup_{n\rightarrow\infty}\int_{\mathbb{R}^{N}}\left(\eta k(x) f(u_{n})u_{n}+\eta\lambda|u_{n}|^{2_{\alpha}^{*}}-V(x)u_{n}^2\right)\varphi_{\varepsilon}\,dx\leq \lambda\nu_{i}.
\end{eqnarray}
By (\ref{31}), we have
\begin{displaymath}
\begin{split}
&\int_{\mathbb{R}^{N}}(-\Delta)^{\frac{\alpha}{2}}u_{n}\cdot(-\Delta)^{\frac{\alpha}{2}}(u_{n}\varphi_{\varepsilon})\,dx\\
=&\iint_{\mathbb{R}^{2N}}\frac{(u_{n}(x)-u_{n}(y))(u_{n}(x)\varphi_{\varepsilon}(x)-u_{n}(y)\varphi_{\varepsilon}(y))}{|x-y|^{N+2\alpha}}\,dxdy\\
=&\iint_{\mathbb{R}^{2N}}\frac{(u_{n}(x)-u_{n}(y))^{2}\varphi_{\varepsilon}(y)}{|x-y|^{N+2\alpha}}\,dxdy
+\iint_{\mathbb{R}^{2N}}\frac{(u_{n}(x)-u_{n}(y))(\varphi_{\varepsilon}(x)-\varphi_{\varepsilon}(y))u_{n}(x)}{|x-y|^{N+2\alpha}}\,dxdy.
\end{split}
\end{displaymath}
It is easy to verify that
$$\iint_{\mathbb{R}^{2N}}\frac{(u_{n}(x)-u_{n}(y))^{2}\varphi_{\varepsilon}(y)}{|x-y|^{N+2\alpha}}\,dxdy\rightarrow \int_{\mathbb{R}^{N}}\varphi_{\varepsilon}\,d\mu,$$
 as $n\rightarrow\infty$ and
$$\int_{\mathbb{R}^{N}}\varphi_{\varepsilon}\,d\mu\rightarrow\mu(\{x_{i}\}),$$
as $\varepsilon\rightarrow0$.
Note that the H\"{o}lder inequality implies
\begin{displaymath}
\begin{split}
&\left|\iint_{\mathbb{R}^{2N}}\frac{(u_{n}(x)-u_{n}(y))(\varphi_{\varepsilon}(x)-\varphi_{\varepsilon}(y))u_{n}(x)}{|x-y|^{N+2\alpha}}\,dxdy\right|\\
\leq&\iint_{\mathbb{R}^{2N}}\frac{|u_{n}(x)-u_{n}(y)|\cdot|\varphi_{\varepsilon}(x)-\varphi_{\varepsilon}(y)|\cdot|u_{n}(x)|}{|x-y|^{N+2\alpha}}\,dxdy\\
\leq&C\left(\iint_{\mathbb{R}^{2N}}\frac{u_{n}^{2}(x)|\varphi_{\varepsilon}(x)-\varphi_{\varepsilon}(y)|^{2}}{|x-y|^{N+2\alpha}}\,dxdy\right)^{1/2}.
\end{split}
\end{displaymath}

Similar to the proof of (\ref{7}), we have $$\lim_{\varepsilon\rightarrow0}\lim_{n\rightarrow\infty}\iint_{\mathbb{R}^{2N}}\frac{u_{n}^{2}(x)(\varphi_{\varepsilon}(x)-\varphi_{\varepsilon}(y))^{2}}
{|x-y|^{N+2\alpha}}\,dxdy=0.$$
Then, combining (\ref{6}) with (\ref{8}), we obtain that  for any $i\in J$,
$$\mu(\{x_{i}\})\leq \lambda\nu_{i}.$$

(2) Taking radially symmetric function $\chi_{R}$ as in Lemma 3.3, we could verify that  $\{u_{n}\chi_{R}\}$ is bounded in $H^{\alpha}_r(\mathbb{R}^{N})$, hence
$$\langle I_{\eta}'(u_{n}),u_{n}\chi_{R}\rangle\rightarrow0,$$
as $n\rightarrow\infty$, which implies
\begin{eqnarray}\label{20}
\begin{split}
\int_{\mathbb{R}^{N}}(-\Delta)^{\frac{\alpha}{2}}u_{n}\cdot(-\Delta)^{\frac{\alpha}{2}}(u_{n}\chi_{R})\,dx
=\int_{\mathbb{R}^{N}}\left(\eta k(x) f(u_{n})u_{n}+\eta\lambda|u_{n}|^{2_{\alpha}^{*}}-V(x)u_n^2\right)\chi_{R}\,dx+o(1).
\end{split}
\end{eqnarray}

Similar to the proof of (\ref{8}),
 we get
\begin{eqnarray}\label{21}
\begin{split}
\limsup_{R\rightarrow\infty}\limsup_{n\rightarrow\infty}\int_{\mathbb{R}^{N}}\left(\eta k(x) f_1(u_{n})u_{n}
+\eta\lambda|u_{n}|^{2_{\alpha}^{*}}
-k(x)f_2(u_n)u_n\right)\chi_{R}\,dx\leq\lambda\nu_{\infty}.
\end{split}
\end{eqnarray}
Notice that
\begin{displaymath}
\begin{split}
&\int_{\mathbb{R}^{N}}(-\Delta)^{\frac{\alpha}{2}}u_{n}\cdot(-\Delta)^{\frac{\alpha}{2}}(u_{n}\chi_{R})\,dx\\
=&\iint_{\mathbb{R}^{2N}}\frac{(u_{n}(x)-u_{n}(y))^{2}\chi_{R}(y)}{|x-y|^{N+2\alpha}}\,dxdy
+\iint_{\mathbb{R}^{2N}}\frac{(u_{n}(x)-u_{n}(y))(\chi_{R}(x)-\chi_{R}(y))u_{n}(x)}{|x-y|^{N+2\alpha}}\,dxdy.
\end{split}
\end{displaymath}
It is easy to verify that
$$\limsup_{R\rightarrow\infty}\limsup_{n\rightarrow\infty}\iint_{\mathbb{R}^{2N}}\frac{(u_{n}(x)-u_{n}(y))^{2}\chi_{R}(y)}{|x-y|^{N+2\alpha}}\,dxdy=\mu_{\infty}$$
and
\begin{displaymath}
\begin{split}
&\left|\iint_{\mathbb{R}^{2N}}\frac{(u_{n}(x)-u_{n}(y))(\chi_{R}(x)-\chi_{R}(y))u_{n}(x)}{|x-y|^{N+2\alpha}}\,dxdy\right|
\leq C\left(\iint_{\mathbb{R}^{2N}}\frac{u_{n}^{2}(x)|\chi_{R}(x)-\chi_{R}(y)|^{2}}{|x-y|^{N+2\alpha}}\,dxdy\right)^{1/2}.
\end{split}
\end{displaymath}
Note that
\begin{displaymath}
\begin{split}
&\limsup_{R\rightarrow\infty}\limsup_{n\rightarrow\infty}\iint_{\mathbb{R}^{2N}}\frac{u_{n}^{2}(x)(\chi_{R}(x)-\chi_{R}(y))^{2}}
{|x-y|^{N+2\alpha}}\,dxdy\\
=&\limsup_{R\rightarrow\infty}\limsup_{n\rightarrow\infty}\iint_{\mathbb{R}^{2N}}\frac{u_{n}^{2}(x)((1-\chi_{R}(x))-(1-\chi_{R}(y)))^{2}}
{|x-y|^{N+2\alpha}}\,dxdy,
\end{split}
\end{displaymath}
then, similar to the proof of (\ref{7}), we obtain
\begin{eqnarray*}\label{12}
\begin{split}
\limsup_{R\rightarrow\infty}\limsup_{n\rightarrow\infty}\iint_{\mathbb{R}^{2N}}\frac{u_{n}^{2}(x)((1-\chi_{R}(x))-(1-\chi_{R}(y)))^{2}}
{|x-y|^{N+2\alpha}}\,dxdy=0.
\end{split}
\end{eqnarray*}
Combining this with (\ref{20}) and (\ref{21}), we have
$$\mu_{\infty}\leq \lambda\nu_{\infty}.$$

\noindent\textbf{Step 2:} There exists $\lambda_{*}>0$ such that for any $0<\lambda<\lambda_{*}$, $\nu_{i}=0$ for any $i\in J$ and $\nu_{\infty}=0$. Suppose that there exists $i_{0}\in J$ such that $\nu_{i_{0}}>0$ or  $\nu_{\infty}>0$, using Lemma 3.4 and Step 1 we obtain
$$\nu_{i_{0}}\leq(S_{\alpha}^{-1}\mu(\{x_{i_{0}}\}))^{2_{\alpha}^{*}/2}\leq(S_{\alpha}^{-1}\lambda\nu_{i_{0}})^{2_{\alpha}^{*}/2}$$
or
$$\nu_{\infty}\leq(S_{\alpha}^{-1}\mu_{\infty})^{2_{\alpha}^{*}/2}\leq(S_{\alpha}^{-1}\lambda\nu_{\infty})^{2_{\alpha}^{*}/2},$$
which implies
\begin{equation}\label{28}
\nu_{i_{0}}\geq(S_{\alpha}\lambda^{-1})^{2_{\alpha}^{*}/(2_{\alpha}^{*}-2)}
\end{equation}
or
\begin{equation}\label{23}
\nu_{\infty}\geq(S_{\alpha}\lambda^{-1})^{2_{\alpha}^{*}/(2_{\alpha}^{*}-2)}.
\end{equation}

By (H3), we have
\begin{eqnarray}\label{27}
\begin{split}
 2I_{\eta}(u_{n})-\langle   I_{\eta}'(u_{n}),u_{n}\rangle
 =&\eta\int_{\mathbb{R}^{N}}k(x)(f(u_n)u_n-2F(u_n))\,dx
 +\eta\lambda\left(1-2/2_{\alpha}^{*}\right)\int_{\mathbb{R}^{N}}|u_n|^{2_{\alpha}^{*}}\,dx\\
 \geq&\eta\lambda\frac{2\alpha}{N}\int_{\mathbb{R}^{N}}|u_n|^{2_{\alpha}^{*}}\,dx
 \geq\lambda\frac{\alpha}{N}\int_{\mathbb{R}^{N}}|u_n|^{2_{\alpha}^{*}}\varphi_{\varepsilon}\,dx.
\end{split}
\end{eqnarray}
Letting $n\rightarrow\infty$,  we obtain that $2C_{3}\geq\lambda\frac{\alpha}{N}\int_{\mathbb{R}^{N}}\varphi_{\varepsilon}\,d\nu$. Since $\int_{\mathbb{R}^{N}}\varphi_{\varepsilon}\,d\nu\rightarrow\nu_{i_{0}}$, as $\varepsilon\rightarrow0$, it follows that
$$2C_{3}\geq\lambda\frac{\alpha}{N}\nu_{i_0}.$$
Similarly, we get $$2C_{3}\geq\lambda \frac{\alpha}{N}\nu_\infty.$$
It follows from (\ref{28}) or (\ref{23}) that
$$2C_{3}\geq\lambda\frac{\alpha}{N}(S_{\alpha}\lambda^{-1})^{2_{\alpha}^{*}/(2_{\alpha}^{*}-2)}
=\frac{\alpha}{N}S_{\alpha}^{N/(2\alpha)}\lambda^{-(N-2\alpha)/(2\alpha)}$$
which implies $\lambda\geq (\frac{\alpha}{2NC_{3}})^{2\alpha/(N-2\alpha)}S_{\alpha}^{2_{\alpha}^{*}/2}\triangleq \lambda_{*}$.

So the assumption  $0<\lambda<\lambda_{*}$
gives a contradiction.
Then, for any $i\in J$, $\nu_{i}=0$ and $\nu_{\infty}=0$.
Using (\ref{14}) we obtain
\begin{displaymath}
\begin{split}
\limsup_{n\rightarrow\infty}\int_{\mathbb{R}^{N}}|u_{n}|^{2_{\alpha}^{*}}\,dx=\int_{\mathbb{R}^{N}}|u_{0}|^{2_{\alpha}^{*}}\,dx.
\end{split}
\end{displaymath}
As
$|u_{n}-u_{0}|^{2_{\alpha}^{*}}\leq2^{2_{\alpha}^{*}}(|u_{n}|^{2_{\alpha}^{*}}+|u_{0}|^{2_{\alpha}^{*}})$,
it follows from the Fatou Lemma that
\begin{displaymath}
\begin{split}
\int_{\mathbb{R}^{N}}2^{2_{\alpha}^{*}+1}|u_{0}|^{2_{\alpha}^{*}}\,dx
=&\int_{\mathbb{R}^{N}}\liminf_{n\rightarrow\infty}(2^{2_{\alpha}^{*}}|u_{n}|^{2_{\alpha}^{*}}+2^{2_{\alpha}^{*}}|u_{0}|^{2_{\alpha}^{*}}-|u_{n}-u_{0}|^{2_{\alpha}^{*}})\,dx\\
\leq&\liminf_{n\rightarrow\infty}\int_{\mathbb{R}^{N}}(2^{2_{\alpha}^{*}}|u_{n}|^{2_{\alpha}^{*}}+2^{2_{\alpha}^{*}}|u_{0}|^{2_{\alpha}^{*}}-|u_{n}-u_{0}|^{2_{\alpha}^{*}})\,dx\\
=&2^{2_{\alpha}^{*}+1}\int_{\mathbb{R}^{N}}|u_{0}|^{2_{\alpha}^{*}}\,dx-\limsup_{n\rightarrow\infty}\int_{\mathbb{R}^{N}}|u_{n}-u_{0}|^{2_{\alpha}^{*}}\,dx,
\end{split}
\end{displaymath}
which implies
$\limsup_{n\rightarrow\infty}\int_{\mathbb{R}^{N}}|u_{n}-u_{0}|^{2_{\alpha}^{*}}\,dx=0$. Then
 $$u_{n}\rightarrow u_{0}\ \mbox{in}\ L^{2_{\alpha}^{*}}(\mathbb{R}^{N}),\ \mbox{as}\ n\rightarrow\infty.$$
 Note that
 $I_{\eta}'(u_{n})\rightarrow0$,
 it follows from (\ref{9}), (\ref{10}) and (\ref{15}) that
 \begin{eqnarray*}
 \begin{split}
 &\limsup_{n\rightarrow\infty}\int_{\mathbb{R}^{N}}|(-\Delta)^{\frac{\alpha}{2}}u_{n}|^2\,dx\\
=&\limsup_{n\rightarrow\infty}\left(\eta\int_{\mathbb{R}^{N}}k(x) f(u_{n})u_{n}\,dx+\eta\lambda\int_{\mathbb{R}^{N}}|u_{n}|^{2_{\alpha}^{*}}\,dx
-\int_{\mathbb{R}^{N}}V(x)u_n^2\,dx\right)\\
\leq&\eta\int_{\mathbb{R}^{N}}k(x) f(u_{0})u_{0}\,dx+\eta\lambda\int_{\mathbb{R}^{N}}|u_{0}|^{2_{\alpha}^{*}}\,dx
-\int_{\mathbb{R}^{N}}V(x)u_0^2\,dx\\
\leq&\liminf_{n\rightarrow\infty}\int_{\mathbb{R}^{N}}|(-\Delta)^{\frac{\alpha}{2}}u_{n}|^2\,dx,
 \end{split}
 \end{eqnarray*}
 which implies
 \begin{equation}\label{16}
\lim_{n\rightarrow\infty}\int_{\mathbb{R}^{N}}|(-\Delta)^{\frac{\alpha}{2}}u_{n}|^2\,dx=\int_{\mathbb{R}^{N}}|(-\Delta)^{\frac{\alpha}{2}}u_{0}|^2\,dx.
 \end{equation}
Thus
  \begin{eqnarray*}
 \begin{split}
 &\limsup_{n\rightarrow\infty}\int_{\mathbb{R}^{N}}V(x)u_n^2\,dx\\
=&\limsup_{n\rightarrow\infty}\left(\eta\int_{\mathbb{R}^{N}}k(x) f(u_{n})u_{n}\,dx+\eta\lambda\int_{\mathbb{R}^{N}}|u_{n}|^{2_{\alpha}^{*}}\,dx
-\int_{\mathbb{R}^{N}}|(-\Delta)^{\frac{\alpha}{2}}u_{n}|^2\,dx\right)\\
=&\eta\int_{\mathbb{R}^{N}}k(x) f(u_{0})u_{0}\,dx+\eta\lambda\int_{\mathbb{R}^{N}}|u_{0}|^{2_{\alpha}^{*}}\,dx
-\int_{\mathbb{R}^{N}}|(-\Delta)^{\frac{\alpha}{2}}u_{0}|^2\,dx\\
=&\int_{\mathbb{R}^{N}}V(x)u_0^2\,dx.
 \end{split}
 \end{eqnarray*}
As $u_n\rightarrow u_0$  weakly in $H^\alpha_r(\mathbb{R}^{N})$, it follows from  (\ref{16}) that
 $$\lim_{n\rightarrow\infty}\int_{\mathbb{R}^{N}}|(-\Delta)^{\frac{\alpha}{2}}(u_{n}-u_0)|^2\,dx+\int_{\mathbb{R}^{N}}V(x)(u_n-u_0)^2\,dx=0,$$
 which implies $u_n\rightarrow u_0$  in $H^\alpha_r(\mathbb{R}^{N})$.
This completes the proof.  \end{proof}

Finally, we will show that the sequence $\{u_n\}$ of critical points for $I_{\eta_{n}}$ is bounded and it is a (PS) sequence for $\widetilde{I}$. Then, from Lemma 3.5 we obtain a nontrivial critical point for  $\widetilde{I}$.
To show the boundedness of $\{u_n\}$,
we will use  the following Pohozaev type
identity  for (\ref{19}):

Let $u\in H^{\alpha}(\mathbb{R}^{N})$ be a weak solution of (\ref{19}), then
\begin{eqnarray}\label{24}
\begin{split}
&\frac{N-2\alpha}{2}\int_{\mathbb{R}^{N}}|(-\Delta)^{\frac{\alpha}{2}}u|^2\,dx
-N\int_{\mathbb{R}^{N}}\left(\eta_{n}k(x)F(u)+\frac{\eta_{n}\lambda}{2_{\alpha}^{*}}|u|^{2_{\alpha}^{*}}-\frac{1}{2}V(x)u^2\right)\,dx\\
=&\int_{\mathbb{R}^{N}}\eta_{n}F(u)\nabla k\cdot x\,dx-\frac{1}{2}\int_{\mathbb{R}^{N}}u^2\nabla V\cdot x\,dx.
\end{split}
\end{eqnarray}
In \cite{Chang2013}, using the $\alpha$-harmonic extension, the authors prove the Pohozaev identity with subcritical nonlinearity. In this paper, although the problem (\ref{19})  involves  critical nonlinearity $|u|^{2_{\alpha}^{*}-2}u$, the potential functions $V(x)$ and $k(x)$,
similar to the proof of Pohozaev identity in \cite{Chang2013}, we could also obtain the Pohozaev identity (\ref{24}), so we do not provide the proof here.

\vspace{2mm}

\noindent \textbf{Proof of Theorem 1.1 } (1) By Theorem 3.1, for almost every $\eta\in[\overline{\eta},1]$, there exists a bounded sequence $\{u_{\eta,n}\}\subset H^{\alpha}_{r}(\mathbb{R}^{N})$ such that $I_{\eta}(u_{\eta,n})\rightarrow c_{\eta}$ and $I_{\eta}'(u_{\eta,n})\rightarrow0$ in
 $H^{\alpha}_{r}(\mathbb{R}^{N})$, as $n\rightarrow\infty$. By Lemma 3.2, $0<c_{\eta}\leq c_0$ for any $\eta\in[\overline{\eta},1]$. We assume that $|I_{\eta}(u_{\eta,n})|\leq c_0+1$ for any $n\in\mathbb{N}$.

  Let $\lambda_*=(\frac{\alpha}{2N(c_0+1)})^{2\alpha/(N-2\alpha)}S_{\alpha}^{2_{\alpha}^{*}/2}$ (see Step 2 in  Lemma 3.5). If $\lambda\in(0,\lambda_*)$, by Lemma 3.5,  passing to a subsequence if possible, there exists $u_{\eta}\in H^{\alpha}_{r}(\mathbb{R}^{N})\setminus\{0\}$ such that
 $u_{\eta,n}\rightarrow u_{\eta}$ in $H^{\alpha}_{r}(\mathbb{R}^{N})$, as $n\rightarrow\infty$. Then, $I_{\eta}(u_{\eta})=c_{\eta}$ and $I_{\eta}'(u_{\eta})=0$.

Let $\{\eta_{n}\}\subset[\overline{\eta},1]$ with $\eta_{n}\uparrow1$ such that there exists $u_n\in H^{\alpha}_{r}(\mathbb{R}^{N})\setminus\{0\}$ satisfying
$I_{\eta_{n}}(u_{n})=c_{\eta_{n}}\leq c_{0}$,   $I_{\eta_{n}}'(u_{n})=0$.
Then $u_{n}$ is a weak solution of the following equation
\begin{equation*}
(-\Delta)^{\alpha}u+V(x)u=\eta_{n}k(x)f_1(u)+\eta_{n}\lambda|u|^{2_{\alpha}^{*}-2}u.
\end{equation*}
By Pohozaev identity for  the above equation, we get
\begin{eqnarray*}
\begin{split}
&\frac{N-2\alpha}{2}\int_{\mathbb{R}^{N}}|(-\Delta)^{\frac{\alpha}{2}}u_n|^2\,dx
-N\int_{\mathbb{R}^{N}}\left(\eta_{n}k(x)F_1(u_n)+\frac{\eta_{n}\lambda}{2_{\alpha}^{*}}|u_n|^{2_{\alpha}^{*}}-\frac{1}{2}V(x)u_n^2\right)\,dx\\
=&\int_{\mathbb{R}^{N}}\eta_{n}F(u_n)\nabla k\cdot x\,dx-\frac{1}{2}\int_{\mathbb{R}^{N}}u_n^2\nabla V\cdot x\,dx.
\end{split}
\end{eqnarray*}
Note that $F(t)\geq0$ for any $t\in\mathbb{R}$, it follows from (V1) and (K2) that
\begin{eqnarray}\label{29}
\begin{split}
\int_{\mathbb{R}^{N}}|(-\Delta)^{\frac{\alpha}{2}}u_n|^2\,dx
=&NI_{\eta_{n}}(u_{n})-\int_{\mathbb{R}^{N}}\eta_{n}F(u_n)\nabla k\cdot x\,dx+\frac{1}{2}\int_{\mathbb{R}^{N}}u_n^2\nabla V\cdot x\,dx\\
\leq& NI_{\eta_{n}}(u_{n})\leq Nc_0.
\end{split}
\end{eqnarray}
Using (\ref{1}), for any $\varepsilon>0$ we obtain
\begin{eqnarray*}
\begin{split}
\int_{\mathbb{R}^{N}}V(x)u_n^2\,dx
&=\eta_{n}\int_{\mathbb{R}^{N}}\left(k(x)f(u_n)u_n+\lambda|u_n|^{2_{\alpha}^{*}}\right)\,dx
-\int_{\mathbb{R}^{N}}|(-\Delta)^{\frac{\alpha}{2}}u_n|^2\,dx\\
&\leq\int_{\mathbb{R}^{N}}\left(k(x)f(u_n)u_n+\lambda|u_n|^{2_{\alpha}^{*}}\right)\,dx\\
&\leq\int_{\mathbb{R}^{N}}\varepsilon k_{2}u_n^2\,dx+\left(k_{2}C(\varepsilon)+\lambda\right)\int_{\mathbb{R}^{N}}|u_n|^{2_{\alpha}^{*}}\,dx.
\end{split}
\end{eqnarray*}
Taking $\varepsilon=V_{1}/(2k_{2})$,  by (\ref{5}) and (\ref{29}) we get
$$\int_{\mathbb{R}^{N}}u_n^2\,dx\leq C\int_{\mathbb{R}^{N}}|u_n|^{2_{\alpha}^{*}}\,dx\leq C\left(\int_{\mathbb{R}^{N}}|(-\Delta)^{\frac{\alpha}{2}}u_n|^2\,dx\right)^{2_{\alpha}^{*}/2}\leq C.$$
Then $\{u_n\}$ is bounded in $H^{\alpha}_{r}(\mathbb{R}^{N})$.  Therefore,  $\{\int_{\mathbb{R}^{N}}k(x)F(u_n)\,dx+\frac{\lambda}{2_{\alpha}^{*}}\int_{\mathbb{R}^{N}}|u_n|^{2_{\alpha}^{*}}\,dx\}$ is bounded. It follows from Remark 3.1 that as $n\rightarrow\infty$
 $$\widetilde{I}(u_n)=I_{1}(u_n)=I_{\eta_{n}}(u_n)+(\eta_{n}-1)\int_{\mathbb{R}^{N}}k(x)F(u_n)\,dx
 +(\eta_{n}-1)\frac{\lambda}{2_{\alpha}^{*}}\int_{\mathbb{R}^{N}}|u_n|^{2_{\alpha}^{*}}\,dx\rightarrow c_1.$$
For any $\phi\in H^{\alpha}_r(\mathbb{R}^{N})$, combining (\ref{1}), the H\"{o}lder inequality with Theorem 2.1 we obtain
\begin{eqnarray*}
\begin{split}
\left|\int_{\mathbb{R}^{N}}(k(x)f(u_n)+\lambda|u_n|^{2_{\alpha}^{*}-2}u_n)\phi\,dx\right|
\leq&\int_{\mathbb{R}^{N}}(k_{2}|u_n\phi|+C|u_n|^{2_{\alpha}^{*}-1}|\phi|)\,dx\\
\leq&k_{2}||u_n||_{L^2(\mathbb{R}^{N})}||\phi||_{L^2(\mathbb{R}^{N})}
+C||u_n||_{L^{2_{\alpha}^{*}}(\mathbb{R}^{N})}^{(2_{\alpha}^{*}-1)/2_{\alpha}^{*}}||\phi||_{L^{2_{\alpha}^{*}}(\mathbb{R}^{N})}\\
\leq&C||\phi||_{H^{\alpha}(\mathbb{R}^{N})}.
\end{split}
\end{eqnarray*}
Since
$$\langle\widetilde{I}'(u_n), \phi\rangle=\langle I_{1}'(u_n), \phi\rangle=\langle I_{\eta_{n}}'(u_n), \phi\rangle+(\eta_{n}-1)\int_{\mathbb{R}^{N}}k(x)f(u_n)\phi\,dx+(\eta_{n}-1)\lambda\int_{\mathbb{R}^{N}}|u_n|^{2_{\alpha}^{*}-2}u_n\phi\,dx,$$
we get as $n\rightarrow\infty $
$$||\widetilde{I}'(u_n)||=\sup\{\big|\langle\widetilde{I}'(u_n),\phi\rangle\big|:||\phi||_{H^{\alpha}(\mathbb{R}^{N})}=1\}\rightarrow 0.$$

 For any $0<\lambda<\lambda_{*}$, passing to a subsequence, still denoted by $\{u_n\}$, we assume that $u_n\rightarrow u_0$ in $H^{\alpha}_{r}(\mathbb{R}^{N})$. Then $\widetilde{I}(u_0)=I_{1}(u_0)=c_1$ and $\widetilde{I}'(u_0)=I_{1}'(u_0)=0$. It follows that $u_0$ is a nontrivial weak solution.

  (2) $u_0$ is nonnegative. In fact, it suffices to consider the following functionals on $H_r^{\alpha}(\mathbb{R}^{N})$:
  $$\widetilde{I}^+(u)=\frac{1}{2}\int_{\mathbb{R}^{N}} (|(-\Delta)^{\frac{\alpha}{2}} u|^{2}+V(x)u^{2})\,dx
-\int_{\mathbb{R}^{N}}k(x)F(u)\,dx-\frac{\lambda}{2_{\alpha}^{*}}\int_{\mathbb{R}^{N}}|u^+|^{2_{\alpha}^{*}}\,dx$$
and
$$I_{\eta}^+(u)=\frac{1}{2}\int_{\mathbb{R}^{N}} (|(-\Delta)^{\frac{\alpha}{2}} u|^{2}+V(x)u^{2})\,dx
-\eta\int_{\mathbb{R}^{N}}k(x)F(u)\,dx-\eta\frac{\lambda}{2_{\alpha}^{*}}\int_{\mathbb{R}^{N}}|u^+|^{2_{\alpha}^{*}}\,dx,$$
  where $u^+=\max\{u,0\}$.

  Similar to the argument of (1), there exists a nontrivial weak solution $u_0$ of (\ref{p}). It is easy to verify that $u_0$ is nonnegative.
  This concludes the proof of Theorem 1.1.
 \qedsymbol

\section*{Acknowledgements}
\noindent
B. Zhang was supported by Natural Science Foundation of Heilongjiang Province of China (No. A201306) and Research Foundation of Heilongjiang Educational Committee (No. 12541667) and Doctoral Research Foundation of Heilongjiang Institute of Technology (No. 2013BJ15).
D. Repov\v{s} was supported in part by the Slovenian Research Agency
grant P1-0292-0101.

\end{document}